\documentclass{article}
\usepackage[utf8]{inputenc}
\usepackage[T1]{fontenc}
\usepackage[british,UKenglish,USenglish,american]{babel}
\usepackage{amsmath}
\usepackage{amsfonts}
\usepackage{amssymb}
\usepackage{hyperref}
\usepackage{mathrsfs}
\usepackage{fancyhdr}
\usepackage{verbatim}
\usepackage{bbm}
\usepackage{bm}
\usepackage{mathtools}
\DeclarePairedDelimiter\floor{\lfloor}{\rfloor}
\usepackage{graphicx}
\usepackage{xcolor}
    \usepackage{booktabs}
    \usepackage{float}

\numberwithin{equation}{section}
\usepackage[titletoc,title]{appendix}
\usepackage{amsthm}
\usepackage{todonotes}

\newcommand{\be}{\begin{eqnarray}}
\newcommand{\ee}{\end{eqnarray}}
\newcommand{\ce}{\begin{eqnarray*}}
\newcommand{\de}{\end{eqnarray*}}
\newtheorem{theorem}{Theorem}[section]
\newtheorem{lemma}[theorem]{Lemma}
\newtheorem{proposition}[theorem]{Proposition}
\newtheorem{corollary}[theorem]{Corollary}
\theoremstyle{remark}

\newtheorem{example}[theorem]{Example}
\newtheorem{remark}[theorem]{Remark}
\newtheorem{definition}[theorem]{Definition}

\usepackage{cleveref}
\Crefname{eqn}{Equation}{Equations}
\Crefname{theorem}{Theorem}{Theorems}
\Crefname{lemma}{Lemma}{Lemmas}
\Crefname{corollary}{Corollary}{Corollaries}
\Crefname{definition}{Definition}{Definitions}
\Crefname{remark}{Remark}{Remarks}
\Crefname{assumption}{Assumption}{Assumptions}
\Crefname{innercustomthm}{Condition}{Conditions}
\crefrangelabelformat{innercustomthm}{#3#1#4-#5#2#6}

\def\R{\mathbb{R}}

\def\N{\mathbb{N}}
\def\Z{\mathbb{Z}}

\def\E{\mathbb{E}}
\def\T{\mathbb{T}}

\def\ZZ{\mathbb{Z}^2 \setminus \{0\}}

\def\bd{\begin{definition}}
\def\ed{\end{definition}}
\def\bp{\begin{proposition}}
\def\ep{\end{proposition}}
\def\bc{\begin{corollary}}
\def\ec{\end{corollary}}
\def\bx{\begin{example}}
\def\ex{\end{example}}

\def\mE{{\mathbb E}}

\def\mR{{\mathbb R}}

\def\geq{\geqslant}
\def\leq{\leqslant}

\allowdisplaybreaks

\newcommand{\dd}{\,\mathrm{d}}

\newcommand{\safeincludegraphics}[2][]{%
  \IfFileExists{#2}{\includegraphics[#1]{#2}}{%
    \fbox{\parbox{0.8\textwidth}{\centering Missing figure file: \texttt{\detokenize{#2}}}}%
  }%
}

\title{Numerical Study of a Surface Growth Model with Singular Noise}
\author{Dirk Bl\"omker, David Buchberger, Johannes  Rimmele}
\date{\today}

\begin{document}

\maketitle

\begin{abstract}
We study a stochastic model for epitaxial thin-film growth driven by spatially rough additive noise in a regime where the noise is singular and regularization via truncation in Fourier space is used to give a meaning to the solution, leading to a vanishing nonlinearity in the limit.
In order to study this phenomenon numerically, the nonlinearity is discretized by a spectral Galerkin projection, while time integration is performed with an exponential Euler scheme. For roughness stronger than space-time white noise, we derive strong error estimates in $L^p(\Omega;C([0,T];\mathcal H^1))$ that display explicitly the interaction between the spatial cut-off, the time step, and the decay of the nonlinear current. 

We also quantify the growth of the truncated stochastic convolution and the corresponding vanishing rate of the nonlinearity. Numerical experiments illustrate the transition from persistent hill formation to noise-dominated dynamics as the roughness parameter increases.

\end{abstract}

\section{Introduction - Physical and Mathematical Model}

We consider the stochastic continuum model for epitaxial thin-film growth studied mathematically in \cite{BR} and, in greater detail, in \cite{Rimmele2026}. The model was introduced for one-dimensional substrates by Hunt, Sander, and coauthors~\cite{HOWOS:94} and subsequently extended to the two-dimensional setting in~\cite{JOHGSSO:94}. General introductions to continuum models for surface growth and molecular beam epitaxy can be found in~\cite{BS:95,KS:95,KS:91,LDS:91}.
One of the principal experimental techniques for producing epitaxial thin films is molecular beam epitaxy. In this process, atoms are deposited from the vapour phase onto a crystalline substrate. After deposition, they diffuse along terraces and eventually settle at energetically favourable sites. This technique is widely used in the fabrication of nanostructures, including quantum wires and quantum dots.

A decisive physical mechanism in epitaxial growth is the Schwoebel barrier~\cite{SCHWOE:66}. This additional energy barrier inhibits deposited atoms from descending across step edges. Consequently, the surface current acquires an effective uphill component, which destabilises flat surface profiles and promotes the formation of mounds and other spatial growth patterns.

The resulting surface evolution is formally described by a phenomenological model given by the stochastic partial differential equation
\begin{equation}\label{e:SPDE}
\partial_t h
=
-\delta \Delta^2 h
-\nabla\cdot\left(
\frac{\nabla h}{1+|\nabla h|^2}
\right)
+\sigma \xi.
\end{equation}
Here, $h=h(t,x)$ denotes the surface height at time $t>0$ above the spatial position $x$. We consider $x\in\T^d$, with $d\in \{1,2 \}$, and impose for simplicity periodic boundary conditions.
The physically most relevant case is $d=2$. The equation is also sometimes considered on the whole space $\R^2$, which leads to substantially different analytical difficulties.

The fourth-order linear term $-\delta\Delta^2h$, with a usually small $\delta>0$, models  linearized surface diffusion and has a regularizing effect on the height profile. 

The  nonlinear term 
describes the divergence of the slope-dependent surface current induced by the Schwoebel barrier. For small slopes, i.e. $|\nabla h|\approx 0$, we have
\begin{align*}
-\nabla\cdot\left(
\frac{\nabla h}{1+|\nabla h|^2}
\right)
\approx -\Delta h,
\end{align*}
hence it acts as a destabilizing uphill diffusion term that would lead to hill growth. 
The competition between this destabilizing current and the stabilizing fourth-order diffusion gives rise to the characteristic pattern formation of the model.

In contrast, for large values of $|\nabla h|$, the magnitude of the surface current vanishes
\begin{align*} 
\frac{\nabla h}{1+|\nabla h|^2}
  \approx 0.
\end{align*}
This limits the destabilizing effect of the nonlinear current for steep surfaces.
As shown in \cite{BR} or \cite{Rimmele2026}, this mechanism also leads to a suppression of the nonlinearity for rough surfaces, which are thus stabilized by the noise. The aim of this paper is to derive error estimates for a numerical discretization that allows this rough noise stabilization mechanism to be resolved numerically.
In the classical formulation of~\eqref{e:SPDE}, the random forcing is given by $\sigma\xi$, where $\sigma>0$ denotes the noise intensity and $\xi=\partial_tW$ is space-time white noise, formally represented as the time derivative of a cylindrical Wiener process $W$. The noise models microscopic fluctuations in the deposition process, such as random variations in the arrival and diffusion of particles.

Our numerical experiments indicate that space-time white noise is not sufficiently rough for this stabilization mechanism to become visible at computationally accessible spatial resolutions. 
In this case, the decay of the nonlinear contribution is too slow, and resolving the
asymptotic regime would require prohibitively fine spatial
discretisations.
For this reason, in the present article,  we consider the more general family  of noises $\xi^{(\alpha)}=\partial_tW^{(\alpha)}$, with $\alpha\in[0,1)$. The case $\alpha=0$ corresponds to classical space-time white noise, whereas $\alpha>0$ yields spatially rougher noise $\xi^{(\alpha)}$. Its precise characterization is given in \eqref{rougher:noise}. 
For all these choices, the equation is considered through a suitable regularization, and the limit of
vanishing regularization leads to the same suppression mechanism of the nonlinear current.  

To describe the surface relative to its mean height, we work in a frame moving with the mean growth and restrict the equation to the mean-zero subspace. Accordingly, we impose
\begin{align*}
\int_{\T^d}h(t,x)\dd x=0,
\qquad t\geq0,
\end{align*}
and remove the spatially constant Fourier mode from the forcing. This eliminates the evolution of the mean height and isolates fluctuations in the surface profile.

We organize the paper in the following way. 
In \Cref{sec:AF} we give the analytic framework in which we will work,
while in \Cref{sec:MR} we state the main result of the present paper. In   \Cref{thm:mainEuler} we derive a detailed error estimate between the solution of the SPDE with regularized noise and the fully discrete approximation.
Moreover, in Theorem \ref{cor:main} we establish convergence of the fully discrete scheme towards the limiting linear SPDE, i.e.\ the Ornstein--Uhlenbeck process. 

In  \Cref{sec:AR} we state estimates on the Ornstein--Uhlenbeck process and divergence rates for vanishing regularization. These are essential to show the result of vanishing non-linearity that allows us to recover the results of \cite{BR} in this setting of rougher noise.

The main result of  \Cref{thm:mainEuler} is proven in Sections \ref{sec:proof} and \ref{sec:mainproof}. 
While \Cref{sec:proof} bounds several error terms, \Cref{sec:mainproof} finishes the proof using a Gr\"onwall-type argument, which allows for poles in time arising from the derivatives in the nonlinearity. 

Finally, \Cref{num:simula} presents numerical experiments that illustrate the different regimes.
 For small values of $\alpha$, corresponding to comparatively more regular noise, the suppression of hill formation is not visible at computationally accessible resolutions. For larger values of $\alpha$, and hence rougher noise, the transition to noise-dominated dynamics becomes clearly observable.   

\section{Analytical Framework}
\label{sec:AF}

In the following numerical analysis, we present simple convergence rates for the approximation of the mild solutions of spectral regularization of the formal SPDE
\begin{equation} \label{e:rSPDE}
	\partial_t u = - \delta \Delta^2 u 
	- \nabla \cdot \frac{\nabla u}{1+|\nabla u|^2} 
	+ \sigma \, \xi^{(\alpha)}, \qquad \text{ for $t \in [0,T]$, $x \in \T^2$,}
\end{equation}
where we consider rougher noise $\xi^{(\alpha)} = \partial_t W^{(\alpha)}$ for $\alpha \in [0, 1)$, and the case $\alpha = 0$ represents classical space-time white noise. 
We also study the corresponding stochastic convolution
	\begin{align} \label {rougher:noise}
Z^{(\alpha)} \in C^0 \left(
	[0,T] , \mathcal{H}^{1- \alpha- \varepsilon } \left(\mathbb{T}^2 \right)
	\right) , \qquad
		Z^{(\alpha)}(t,x) \coloneqq  \sum_{k \in \ZZ} \mu_k^{\frac{\alpha}{2}} Z_k(t) e_k(x), 
	\end{align}
	for each $\varepsilon >0$ and $t \in [0,T]$, $x \in \T^2$,
	with 
    \[Z_k(t)
\coloneqq
\int_0^t e^{-  (t-s) \delta\mu_k^2}\,\dd\beta_k(s)
\quad \text{for each }k \in \ZZ,
\]
according to the eigenvalues $\left(\mu_k \right)_{k \in \ZZ}$ of $-\Delta$ (see \eqref{Laplace:EV}),
    where $\left(\beta_k \right)_{k \in \ZZ}$ is an i.i.d.\ family of Brownian motions.
    Throughout, we fix $\sigma>0$. 
    
    For simplicity, we choose $\delta = 1$. 
    This means, we examine the SPDE 
    \eqref{e:rSPDE} with the negative Bilaplace operator~$A \coloneqq -\Delta^2$ multiplied by the parameter~$\delta = 1$.
    Any other choice would only change the precise value of the constants. 

    \begin{remark}
    From a numerical point of view, the consideration of rougher noise as in \eqref{e:rSPDE} makes sense due to the only logarithmic convergence of 
    $\E\big[\big \|P_{N} Z^{(0)}(t)\big \|_{\mathcal{H}^1}^2\big] $ 
    for the spectral projection $P_N$ introduced below in \eqref{def:PN} (see \Cref{convrate:Zalpha}).
    The speed of this divergence determines how fine we need to discretize in order to see the effect of vanishing nonlinearity.
    \end{remark}

    \begin{remark}
    Let $\alpha, \beta  > 0$   and $d \in \N$ be the dimension.
 Then the eigenvalues $\big( - \mu_k^{\beta} \big)_{k \in \Z^d \setminus \{0\} }$ and the corresponding eigenfunctions $\left(e_k  \right)_{k \in \Z^d \setminus \{0\} }$ of the linear operator $- (-\Delta)^{\beta}$ on the $d$-dimensional torus $\T^d$ satisfy 
    $$
    0 \leq \mu_k^{\beta} \asymp |k|^{2\beta},
    $$
    where $\asymp$ denotes an estimate from below and above for two different constants.
    
    Thus, it is well known to show that 
    the stochastic convolution 
    $$
    Z_d^{(\alpha)} (t,x)
    \coloneqq \sum_{k \in \Z^d \setminus \{0\} } \mu_k^{\frac{\alpha}{2}} \int_0^t  e^{-(t-s)\mu_k^{\beta}} \dd \beta_k(s) e_k(x) , \qquad
    t \in [0,T] , \ x \in \T^d
    $$  
    has the regularity
    \begin{align} \label {d-dim:rougher:noise}
Z_d^{(\alpha)} \in C^0 \left(
	[0,T] , \mathcal{H}^{\rho - \alpha  } \left(\mathbb{T}^d \right)
	\right),
    \qquad
    \text{ for $\rho < \beta  - \frac{d}{2}$.}
	\end{align}
    For a definition of the standard Sobolev-space $\mathcal{H}^s$ see Definition \ref{def:Halpha}.
    \end{remark}
    
	We will use the following notation for the nonlinearity
    \begin{align*}
        f(z)\coloneqq\frac{z}{1+|z|^2}, \quad \text{ for $z \in \R^2$,}
\qquad
\mathfrak F(u)\coloneqq-\nabla\cdot f(\nabla u) \quad \text{for } u\in\mathcal{H}^1.
    \end{align*}
	Since the focus is on analysing convergence rates with respect to the spatial and temporal discretisation, the dependence of constants on $L$ and $T$ will not always be tracked explicitly in this context.

	This paper employs a spectral Galerkin method for spatial discretisation and an Euler scheme for temporal integration, as detailed in the following.
    Note that we consider a standard interpolation in time of the mild formulation using a rounding down to the temporal grid in the nonlinearity.
    
	\begin{definition}[Spectral Galerkin method] \label{def:SGm}
		
		Let $n \in \N$ and define $h \coloneqq \frac{T}{n}$. The spectral Galerkin method with exponential Euler in time is then given by
		\begin{align} \label {eq:Euler-uN}
			u_h^{(N)}(t)
			\coloneqq 
			e^{tA}P_N u_0 +\int_0^t e^{(t-s)A} P_N \mathfrak{F} (u_h^{(N)}(\floor{s}_h)) \dd s 
			+ \sigma P_N Z^{(\alpha)}(t),
		\end{align}
		where the time variable $s$ is rounded down to the nearest grid point $jh$,
		for $j=0,1,\ldots, n -1$,
		if $s \in \left[j h ,\ (j+1) h \right)$,
		and we use the
		orthogonal projection
		\begin{align} \label{def:PN}
			P_N : L^2 \left(\mathbb{T}^2 \right) \to L^2 \left(\mathbb{T}^2 \right),
			\quad
			P_N
			\sum_{k \in \ZZ } u_k e_k(x) =
			\sum_{k \in \ZZ , |k| \leq N } u_k e_k(x).
		\end{align}
	\end{definition}

	\begin{definition} [Exponential Euler scheme] \label[definition]{def:exponential-Euler}
		Evaluating the mild solution at the grid points $jh$ yields the fully discrete exponential Euler scheme for $v^{(N)}_j = u_h^{(N)}(jh)$.
		\[
		v^{(N)}_{j+1} =
		e^{hA}v^{(N)}_{j} +
		\int_0^h e^{(h-s)A} P_N \mathfrak{F} (v^{(N)}_j) \dd s + \sigma P_N \tilde{Z}_j,\quad j=0,1,\ldots, n - 1
		\]
		with
		\[
		\tilde{Z}_j = \int_{jh}^{(j+1)h} e^{\left((j+1)h-s \right)A}\dd W^{(\alpha)}(s).
		\]
	\end{definition}
        Furthermore, we define the mild solution of the regularized SPDE 
        \begin{align} \label{def:uN}
			u^{(N)} (t) & \coloneqq 
			e^{tA} u_0 + \int_{0}^t e^{(t-s)A}  \mathfrak{F} \left(
			u^{(N)}(s)
			\right) \dd s + \sigma P_N Z^{(\alpha)}(t)
		\end{align}
        for $t \in [0,T]$ and  $N \in \N$.
        
		\noindent
		For $\alpha=0$ and the cut-off sequence
		$\left(\alpha^{(N)}_k\right)_{k\in \ZZ } =
		\left(\mathbbm{1}_{|k| \leq N } \right)_{k\in \ZZ }
		$
		as a regularizing sequence for the white noise,
		i.e.
		\[
		P_N Z^{(0)}(t) = \sum_{k \in \ZZ} \alpha_k^{(N)} Z_k(t) e_k(x).
		\]
		In \cite{BR} we have shown
		\begin{align}  
			u^{(N)} (t) 
			\to
			u(t)=e^{tA} u_0 + \sigma Z^{(0)}(t) 
            \quad \text{for } N\to\infty
		\end{align}
		in $L^p \left(\Omega , C^0 \left( [0,T] \times \T^2 \right) \right)$ for any $p>1$, but with a very slow logarithmic rate. 
        
		In particular we saw in \cite{BR}, that the limiting function $u$ of the sequence $\left(u^{(N)}\right)_{N \in \N}$ is independent of the regularization and $u$ is the mild solution of
		\begin{align*}  
			\begin{cases}
				\partial_t u & = -\Delta^2 u + \sigma \partial_t W, \\
				u (0) & = u_0.
			\end{cases}
		\end{align*}
        However, we will briefly revisit these results in the following, as for $\alpha\not=0$ we obtain the mild solution of
        \begin{align} \label{eq:SPDE-ou}
			\begin{cases}
				\partial_t u & = -\Delta^2 u + \sigma \partial_t W^{(\alpha)}, \\
				u (0) & = u_0.
			\end{cases}
		\end{align}

	\begin{remark}
	Throughout this article, the cut-off sequence
	$$\big(\alpha^{(N)}_k \big)_{k\in \ZZ } =
	\big(\mathbbm{1}_{|k| \leq N } \big)_{k\in \ZZ }$$
	can be replaced
	with $\big(\mathbbm{1}_{|k| \leq \zeta(N) } \big)_{k\in \ZZ }$, where $\left( \zeta(N) \right)_{N \in \N}$ is any non-decreasing sequence that diverges to infinity as $N$ increases.	
	Equivalently, this can be characterized by an orthogonal projection	
	\begin{align*}	
		P_{\zeta(N)} : L^2 \left(\mathbb{T}^2 \right) \to L^2 \left(\mathbb{T}^2 \right),		
		\quad		
		P_{\zeta(N)}		
		\sum_{k \in \ZZ } u_k e_k(x) 
		=
		\sum_{k \in \ZZ , |k| \leq \zeta(N) } u_k e_k(x).	
	\end{align*}
	However, for the sake of simplicity, we will focus our investigation on the case of \Cref{def:SGm} for $\alpha \in (0,1)$.
	\end{remark}

    Throughout, $C>0$ denotes a generic constant that may change from line to line. Its dependence on the relevant parameters is specified in the statements of the respective results and is generally omitted from the proofs.
    The dependence on the parameters $L$ and~$T$ is completely omitted.    
	
		\section{Main Results} \label{sec:MR}
        
		This section presents the main results of the paper, beginning with the convergence rate of the error function $e^{(N)}_h  \coloneqq  u^{(N)}-u_h^{(N)}$ when comparing the mild solution $u^{(N)}$ of the SPDE with regularized noise to the Euler discretisation of the spectral Galerkin approximation $u_h^{(N)}$.
		
		\begin{theorem}  \label {thm:mainEuler} 
				Let $T>0$,  $\gamma \in (0,2)$, $p > \frac{4}{\gamma}$, $\varepsilon \in \big( 0 , \frac{1}{2} - \frac{1}{p} \big)$, $\alpha \in (0,1)$,    and $u_0 \in \mathcal{H}^{1 + \eta} \left(\T^2 \right) \cap W^{1 , \infty} \left(\T^2 \right)$ with~$\eta \in (0,4)$.
			Let $u^{(N)}_h$ and~$u^{(N)}$ be defined by \eqref{eq:Euler-uN} and \eqref{def:uN}.
			Then, there is a constant  $C = C\left(\alpha , \eta , p , \gamma, \varepsilon, \sigma, L, T, \left \| u_0 \right \|_{W^{1 , \infty}} \right) >0$, independent of $N$ and $h$, such that we have
			\begin{align*} 
				& \left( \mE \sup_{t\in[0,T]} \left\|u_h^{(N)}(t)-u^{(N)}(t) \right \|_{\mathcal{H}^1}^p \right)^{\frac{1}{p}} \nonumber
				\\& \leq
				C \left( 
				\| u_0\|_{\mathcal{H}^{1+\eta }} \left( h^{\frac{\eta }{4}} +
				N^{-\eta }
				\right)
				+
                N^{2  + \alpha }
				h^{\frac{1}{2} - \frac{1}{p}  - \varepsilon}
				+ 
				N^{-2 + \gamma - \frac{2 \alpha}{p+2}}
				+
				{h}^{\frac{1}{2} - \frac{\gamma}{4}}
				\left[ N^{- \frac{2 \alpha}{p +2}} + h \right] 
				\right).
			\end{align*}
		\end{theorem}
		
		\noindent
		Similarly to the vanishing nonlinearity result (see \cite[Theorem~3.2]{Rimmele2026}), we present a related theorem demonstrating that the Euler discretisation of the spectral Galerkin approximation $u_h^{(N)}$ converges to the OU-process $u$. In the limit, the function does not exhibit linear instability, and the surface remains without a rough hill structure.
		\begin{theorem}
			 \label {cor:main}
		Let $T>0$,  $p \geq 1$, $\varepsilon > 0 $, $\alpha \in (0,1)$,   and $u_0 \in \mathcal{H}^{1 + \eta} \left(\T^2 \right) \cap W^{1 , \infty} \left(\T^2 \right)$ with $\eta \in (0,4)$.
		Let $u^{(N)}_h$, as defined in \eqref{eq:Euler-uN}, and $u$ denote the mild solution to \eqref{eq:SPDE-ou}.
		Then, we have
			\begin{align}
				 \label {est:main-cor}
                 \lim_{N \to \infty}
				  \lim_{h\rightarrow0}  \mE \sup_{t\in[0,T]} \left\|u_h^{(N)}(t)-u(t) \right\|_{\mathcal{H}^{1- \alpha - \varepsilon}}^p=0 .
			\end{align}
		\end{theorem}
		
		\begin{proof}[Sketch of Proof]
        By the triangle inequality, \Cref{thm:mainEuler}, and the vanishing of the nonlinearity from \Cref{cor:vanolin}, 
            we obtain  \eqref{est:main-cor} for $p > \frac{4}{\gamma}$ such that~$\frac{p}{p-1} \in \big(1 , \frac{1}{1-\frac{\gamma}{4}} \big)$.
			Applying Hölder's inequality, this establishes the desired result for each $p \geq 1$.
		\end{proof}

        Note that \Cref{cor:vanolin} is the key to extend the convergence of $u^{(N)}$ to $u$ from the case $\alpha=0$ treated in \cite{BR} 
        to the case of positive $\alpha$ treated here.
        This is straightforward from the mild formulation and verbatim as in \cite{BR}, so we do not give the details here. 
        
       Let us also remark that we could derive a rate in  \Cref{est:main-cor}, but as this is very poor for small $\alpha$ we skip details here.
     We also believe that it would be possible to improve the rate in \Cref{thm:mainEuler} using techniques  like stochastic Sewing Lemma, but in as the rate in \Cref{est:main-cor} is poor, we only present a significantly simpler proof here. 
		
		\section{Auxiliary Results} \label{sec:AR}

		In this section, we present auxiliary results to prove \Cref{thm:mainEuler} and \Cref{cor:main}.
		Specifically, we demonstrate the divergence rate of stochastic convolution in \Cref{convrate:Zalpha} to prove the convergence rate for the vanishing nonlinearity in \Cref{Conva:Ratealpha}, and the Hölder continuity in time in \Cref{lem:Z-beta}.
		\begin{lemma}  \label[lemma]{convrate:Zalpha}
			The stochastic convolution exhibits the following growth estimate:
			\begin{align*}
				\E
				\left \|
				P_{N} Z^{(\alpha)}(t)
				\right \|_{\mathcal{H}^1}^2
				& \asymp 
				\begin{cases}
					\ln \left(N\right), &  \text{if $\alpha =0$} \\
					N^{2 \alpha}, &  \text{if $\alpha \in (0,1)$}
				\end{cases}
				\end{align*}
			for each $t \in (0,T]$.
		\end{lemma}

        Let us remark that we would obtain a uniform in $t\in[0,T]$ upper bound in \Cref{convrate:Zalpha}, but no uniform lower bound as $P_{N} Z^{(\alpha)}(0)=0$
		 
		\begin{proof}
			For each $t \in (0,T]$, we derive the following growth behavior by applying It\^o's isometry and using the independence of the family $\left(\beta_k \right)_{k \in \ZZ}$ in the first equation below
			\begin{align*}
				\E
				\left \|
				P_{N} Z^{(\alpha)}(t)
				\right \|_{\mathcal{H}^1}^2
				&	=
				\sum_{ 0 < |k| \leq N}
				\mu_k^{1+\alpha}
				\int_0^t e^{-2(t-s) \mu_k^2} \dd s 
                =
				\sum_{ 0 < |k| \leq N}
				\frac{\mu_k^{\alpha}}{2 \mu_k} 
				\left[1- e^{-2t \mu_k^2} \right] \nonumber
				\\
				&	
                \asymp 
                \sum_{ 0 < |k| \leq N}
				|k|^{2\alpha -2}
                \asymp 
				\int_1^{N}
                x^{2 \alpha - 1}
				  \dd x  
                \asymp 
				\begin{cases}
					\ln \left(N\right), &  \text{if $\alpha =0$} \\
					N^{2 \alpha}, &  \text{if $\alpha \in (0,1)$}.
				\end{cases}
			\end{align*}

		\end{proof}

			\begin{lemma} \label[lemma]{lem:Z-beta}
			For any $p\geq1$, $\eta \in(0,1]$, and $\psi \in (0,1]$ there exists a constant $C = C(p,T, \eta , \psi, L , \alpha) >0$ such that for any $0\leq s\leq t\leq T$ the following  bounds hold
			\begin{align*} 
				 \left(\mE  \left \| P_{N} Z^{(\alpha)}(t)-P_{N} Z^{(\alpha)} (s) \right \|^p_{C^0 \left(\T^2 \right)}  \right)^\frac{1}{p}
				& \leq
				C  |t-s|^{\frac{\psi}{2}} N^{\alpha + \eta + 2 \psi   } ,
				\\
				 \left(\mE \left \| P_{N} Z^{(\alpha)}(t)-P_{N} Z^{(\alpha)} (s) \right \|^p_{C^1 \left(\T^2 \right)} \right)^\frac{1}{p}
				& \leq 
				C  |t-s|^{\frac{\psi}{2}} N^{\alpha + \eta + 2 \psi +1 },
                \\
                \left(\mE  \left \| P_{N} Z^{(\alpha)}(t)-P_{N} Z^{(\alpha)} (s) \right \|^p_{\mathcal{H}^1 \left(\T^2 \right)} \right)^\frac{1}{p}
				& \leq
				C  |t-s|^{\frac{\psi}{2}} N^{\alpha   + 2 \psi   } .
			\end{align*}
		\end{lemma}

		\begin{proof}
        Arguing as in the proof of \cite[Theorem~2.21]{Rimmele2026} (see \Cref{lem:Z-C0} for the result), we obtain  that for each $\psi \in (0,1]$ there is a constant $C_\psi>0$ such that
			\begin{align*}
				\mathbb{E}
				\left| 
				Z_k(t) - Z_k(s)
				\right|^2
				\leq
				C_\psi
				\mu_k^{2\psi - 2} | t-s |^\psi
			\end{align*}
			holds for each $t,s \in [0,T]$.
			Furthermore, for $p \geq 1$, we derive from Gaussian hypercontractivity (see \cite[Theorem~2.25]{Rimmele2026}) the first upper bound below, and thus obtain 
			\begin{align*}
					\mE  \left \| P_{N} Z^{(\alpha)}(t)-P_{N} Z^{(\alpha)} (s) \right \|^p_{C^0} 
				& \leq C_{p , \eta} \left(
				\mE  \left \| P_{N} \left( Z^{(\alpha)}(t) - Z^{(\alpha)}(s) \right) \right \|^2_{\mathcal{H}^{1 + \eta}}  \right)^\frac{p}{2} 
				\\ 
				& \leq C_{p , \eta}
				\left(
				\sum_{ 0 < |k| \leq N} \mu_k^{\alpha + \eta + 1} \ 
				\mathbb{E}
				\left| 
				Z_k(t) - Z_k(s)
				\right|^2
				\right)^\frac{p}{2}
				\\ 
				& \leq C_{\psi, p , \eta}
				\left(
				\sum_{ 0 < |k| \leq N}
				\mu_k^{\alpha + \eta + 2\psi - 1} | t-s |^\psi
				\right)^\frac{p}{2}
				\\ 
                & \leq C_{L, \psi, p , \eta}
				\left(
				\sum_{ 0 < |k| \leq N}
				|k|^{2 \alpha + 2\eta + 4\psi - 2} | t-s |^\psi
				\right)^\frac{p}{2}
				\\ 
				& \leq C_{L, \psi, p , \eta}
				\left(
				| t-s |^\psi \Big(1+
				\int_1^{N} 
				\tau^{2 \alpha + 2\eta + 4\psi - 1} \dd \tau \Big)
				\right)^\frac{p}{2}
				\\ 
				& \leq C_{L, \psi, p , \eta}
				| t-s |^{\frac{\psi}{2}p}
				N^{(\alpha + \eta + 2 \psi)p   },
			\end{align*}
            where we used a comparison principle for Riemannian integrals and sums together with polar coordinates. 
			
			Similarly as above, we directly derive
			\begin{align*}
					\mE  \left \| P_{N} Z^{(\alpha)}(t)-P_{N} Z^{(\alpha)} (s) \right \|^p_{C^1}  
				& \leq C_{p , \eta} 
				\left( \mE  \left \| P_{N} \left( Z^{(\alpha)}(t) - Z^{(\alpha)}(s) \right) \right\|^2_{\mathcal{H}^{2+ \eta}} \right)^\frac{p}{2} 
				 \\ & \leq C_{\psi, p , \eta}
				| t-s |^{\frac{\psi}{2} p}
				N^{(\alpha + \eta + 2 \psi + 1)p}.
			\end{align*}
Similarly, we obtain the asserted $\mathcal{H}^1$-estimate.
		This confirms the claim.
		\end{proof}

		  \subsection{Convergence Rate for the Rougher Noise Case}
		In this subsection we consider the rate for the vanishing of the nonlinearity in the case of $\alpha \in (0,1)$, i.e. noise that is rougher than the regular white noise.
		Recall~$Z^{(\alpha)}$ from \eqref{rougher:noise} and $u^{(N)}$ from \eqref{def:uN}.

		\begin{lemma}  \label[lemma]{Conva:Ratealpha}
       Let $T>0$, $p\geq1$, $\alpha\in(0,1)$, and $u_0\in
        \mathcal H^{1+\eta}(\T^2)
        \cap W^{1,\infty}(\T^2)$ for $ \eta\in(0,4)$.
       For $N\in\mathbb N$, define $$t_N\coloneqq N^{-\frac{2\alpha p}{p+2}}. 
       $$ 
			Let  $f(z)\coloneqq \frac{z}{1+|z|^2}, z\in\mR^2
            $.
            Then, there is a constant $C = C\left(\alpha , p , \sigma, L, T, \left \| u_0 \right \|_{W^{1 , \infty}} \right) >0$  and $N_0\in\mathbb N$ such that,  for all sufficiently large $N\geq N_0$, we have
    \begin{align}
        \sup_{t\in[t_N,T]}
        \sup_{x\in\T^2}
        \mathbb E
            \left|
            f\left(\nabla u^{(N)}(t,x)\right)
            \right|^p
        \leq
        C N^{-\frac{2\alpha p}{p+2}}.
        \label{est:nonlinear-no}
    \end{align} 
		\end{lemma}
        
			\begin{proof}
            For this proof, we assume without loss of generality that $\sigma=1$. 
			This proof is based directly on the facts that the covariance operators of $P_N Z (t,x)$ and $\nabla P_N Z (t,x)$ do not depend on $x \in \T^2$ (see \cite[Lemma~3.10]{Rimmele2026}), and 
            that nonlinearity vanishes
            	\begin{align*}
			\sup_{x \in \T^2 } 
			\mathbb{E}  \left | 
			f \left(\nabla u^{(N)} (t,x) \right) 
			\right |^p   \xrightarrow{N \to \infty}  0   \qquad
            \text{ for each $p \geq 1$ and $t \in (0,T]$,}
		\end{align*}
            (see \cite[Theorem~3.11]{Rimmele2026}).
			In \cite[Theorem~3.9]{Rimmele2026} it is shown that there is a uniform bound $\hat{M}>0$ such that 
            $$ \left | \nabla v^{(N)} (\omega , t , x)  \right | \leq \hat{M},
            \qquad \text{ for each } \omega \in \Omega , \ t \in [0,T] , \ x \in \T^2,
            $$
            where 
            $$
            v^{(N)}  \coloneqq 
            u^{(N)}  -  \sigma   P_{N} Z^{(\alpha)}
            - e^{t A} u_0.
            $$
            However, for $u_0 \in W^{1 , \infty }$ there exists a constant $\tilde{M} >0$ such that
            $
            \sup_{t \in [0,T]} \left \| \nabla  e^{t A} u_0  \right \|_{\infty} \leq \tilde{M}
            $ 
            and thus we obtain
            $$ \left | \nabla u^{(N)} (\omega , t , x)  -  \sigma  \nabla P_{N} Z^{(\alpha)} (\omega , t , x)  \right | \leq \hat{M} + \tilde{M} \eqqcolon M,
            \qquad \text{ for each } \omega \in \Omega , \ t \in [0,T] , \ x \in \T^2.
            $$
            Set
			$\vartheta \coloneqq \frac{1}{p+2} \in \left(0, \frac{1}{2} \right)$
			and
			$K_{N} \coloneqq \max \{ N^{2\alpha \vartheta} ,  2 M \}.
			$
			Furthermore,  from \cite[Lemma~3.10]{Rimmele2026}, we have that the covariance matrix 
			$$
			\Sigma_{N}(t,x)
			\coloneqq  \operatorname{Cov} \left(
			\frac{ \nabla P_{N} Z^{(\alpha)}(t,x)}{K_{N}  }
			\right)
			$$
			does not depend on the variable~$x \in \T^2$.
            We define
            $$
            \mathcal{R}_N \coloneqq \left \{
            k \in \ZZ \ : \ \frac{N}{4} \leq k_1 , k_2 \leq \frac{N}{2}
            \right \} \subset \left \{
            k \in \ZZ \ : \  | k | \leq N
            \right \}.
            $$ 
            Note that $\#  \mathcal{R}_N \asymp N^2$ and  for every $k \in \mathcal{R}_N$ we have 
            \begin{align*}
                \frac{	\mu_{k}^{\alpha} k_1^2}{|k|^4} \gtrsim N^{2 \alpha -2} \qquad \text{ and } 
                \qquad
                \mu_{k}^{2}   
                 \gtrsim N^{4}.
            \end{align*} 
        Furthermore, because of   $\alpha\in(0,1)$ and $p/(p+2)<1$, we have
        $
        \beta\coloneqq
        4-\frac{2\alpha p}{p+2}>2
        $.  Hence, for
        $t\in[t_N,T]$ and $k\in\mathcal R_N$,
        \[
        1-e^{-2t\mu_k^2}
        \geq
        1-e^{-cN^\beta}
        \geq
        1-e^{-c}
        \eqqcolon c_0>0.
        \] 
            Thus, we obtain, for each $t \in [t_N , T]$, 
            	\begin{align*}
				  \det 
				\left(
				\Sigma_{N}(t,x)
				\right) & 
				\gtrsim
				\frac{1}{\left(K_{N} \right)^4}
				\left[
				\sum_{k \in \ZZ,  |k| \leq {N}}
				\frac{	\mu_{k}^{\alpha} k_1^2}{|k|^4}
                \left[ 
                1 - e^{-2 t \mu_k^2}
                \right]
				\right]^2
				\\ &
				\gtrsim
				\frac{1}{\left(K_{N} \right)^4}
				\left[
				\sum_{k \in \mathcal{R}_N }
				\frac{	\mu_{k}^{\alpha} k_1^2}{|k|^4}
                \left[ 
                1 - e^{-2 t \mu_k^2}
                \right]
				\right]^2
				\\ &
                \gtrsim
				 N^{-8 \alpha \vartheta}
				\left[
				\sum_{k \in \mathcal{R}_N }
				N^{2 \alpha -2}
                \left[ 
                1 - e^{-2 c N^{4  - \frac{2 \alpha p}{p+2} } }
                \right]
				\right]^2
                \\ &
                \gtrsim
				 N^{-8 \alpha \vartheta}
				\left[
				\sum_{k \in \mathcal{R}_N }
				N^{2 \alpha -2} 
				\right]^2
                 \\ &
                \gtrsim
				 N^{4 \alpha -8 \alpha \vartheta}  .
			\end{align*}

			The nonlinearity thus satisfies, uniformly for each $t \in [t_N ,T]$ and $x \in \T^2$,  the following upper bound 
			\begin{align*}
				& \sup_{x \in \T^2 }   
				\mE  \left |
				f \left(
				\nabla u^{{N}} (t,x)
				\right)
				\right |^p
				\\	& \leq
				\sup_{x \in \T^2 }   
				\mE \left[ \left |
				f \left(
				\nabla u^{{N}} (t,x)
				\right)
				\right |^p   \right |
				\ \left . 
				\left | 
				\nabla P_{N} Z^{(\alpha)} (t,x)
				\right | > K_{N} 
				\right] \nonumber
				\\ & \quad +
				\sup_{x \in \T^2 }   
				\mE \left[ \left |
				f \left(
				\nabla u^{{N}} (t,x)
				\right)
				\right |^p \right | 
				\ \left . 
				\left | 
				\nabla P_{N} Z^{(\alpha)} (t,x)
				\right | \leq  K_{N} 
				\right] 
				\mathbb{P}
				\left(
				\left | 
				\nabla P_{N} Z^{(\alpha)} (t,x)
				\right | \leq  K_{N}
				\right)
				\nonumber
				\\ &
				\leq
				\left( \frac{2}{1 + K_{N}  - M} \right)^p 
				+ C \max_{y \in \overline{B_1(0) }} \{ \varphi_N (t,x, y) \} 
				\\ &
				\lesssim
				\left( \frac{4}{ K_{N} } \right)^p 
				+
				\frac{1}{\sqrt{\det \left(\Sigma_{N} (t,x) \right)}}  \nonumber
                \\ &
				\lesssim
				N^{-2 \alpha p \vartheta}
				+
				\frac{1}{\sqrt{\det \left(\Sigma_{N} (t,x) \right)}} , \nonumber
			\end{align*}
			whereby 
			$$\varphi_{N}(t,x,y) \coloneqq
			\frac{1}{2 \pi \sqrt{\det \left(
					\Sigma_{N}(t,x)
					\right)}}
			e^{
			- \frac{1}{2} y^T  
			\Sigma^{-1}_{N}(t,x)
			 y }
			$$ is the density function of $
			\frac{ \nabla P_{N} Z^{(\alpha)} (t,x)}{K_{N} }
			\sim \mathcal{N} \left(
			\mathbf{0} , \Sigma_{N}(t,x)
			\right)
			$.
			By the choice $\vartheta = \frac{1}{p+2}$, we obtain
            \begin{align*}
                \sup_{t\in[t_N,T]}
                \sup_{x\in\T^2}
                \mathbb E
                \left|
                f\left(\nabla u^{(N)}(t,x)\right)
                \right|^p
				&
				\lesssim
				N^{-2 \alpha p \vartheta}
				+
				N^{2\alpha \left(2 \vartheta  -1\right)}  
				\lesssim
				{N}^{- \frac{2\alpha p}{p+2}}.
			\end{align*} 
            In particular, for each $t \in [t_N , T]$ with $\floor{t}_h \geq t_N$ the same statement applies to  the nonlinearity of $u_h^{(N)}$.
		This shows the assertion.
		\end{proof}
	
		\begin{remark}
		For a more general nonlinearity $f$ satisfying
		\begin{align*}
			| f \left(\nabla u \right) | \asymp | \nabla u|^{-\eta} 
		\end{align*}
		for $\eta >0$, we even get the bound
		\begin{align*}
			\sup_{x \in \T^2 }   
			\left( \mE  \left |
			f \left(
			\nabla u^{(N)}(t,x)
			\right)
			\right |^p
			\right)^{\frac{1}{p}}
			\lesssim
			{N}^{- \frac{2\alpha \eta }{p \eta +2}}
		\end{align*}
	for each $t \in (0,T]$, by choosing $\vartheta \coloneqq \frac{1}{p \eta +2}$ and applying the inequality
	$$
		\sup_{x \in \T^2 }   
	\mE \left[ \left |
	f \left(
	\nabla u^{{N}} (t,x)
	\right)
	\right |^p   \right |
	\ \left . 
	\left | 
	\nabla P_{N} Z^{(\alpha)} (t,x)
	\right | > K_{N} 
	\right]  
	<
	 \left | K_N - M \right|^{- p \eta}
	$$
	as in the proof of \Cref{Conva:Ratealpha}.
	\end{remark}

		\begin{lemma} \label[lemma]{cor:vanolin}
			For each $p>1$ there exists a constant $C = C\left(\alpha , p , \sigma, L, T, \left \| u_0 \right \|_{W^{1 , \infty}} \right) >0$, such that
			\begin{align*}
				\left( \mE  \left \| 
				f \left(\nabla u^{(N)}  \right) 
				\right \|^p_{L^p \left( [0,T]\times \T^2 , \R^2\right) } \right)^{\frac{1}{p}}
				\leq C
				{N}^{- \frac{2 \alpha}{p+2}}.
			\end{align*}

             In particular,  the same convergence rate holds for the nonlinearity of $u_h^{(N)}$.
		 \end{lemma}

    Note that as in \cite{BR} this result is used to show that $u^{(N)}$ converges to $u$ using the mild formulation.

		\begin{proof}
			By Tonelli's theorem and the pointwise bound on $f$ from \Cref{Conva:Ratealpha} for $t \geq t_N$ as well as the uniform boundedness $\| f\|_{L^\infty\left( \R^2 , \R^2\right) } \leq 1$, we obtain the following bound
			\begin{align*}
				\left( \mE  \left \| 
				f \left(\nabla u^{(N)}  \right) 
				\right \|^p_{L^p \left( [0,T]\times \T^2 , \R^2\right) } \right)^{\frac{1}{p}}
				 & 
				=
				\left( \mE 
				\int_{[0,L]^2} \int_0^T 
				\left |
				f \left(\nabla u^{(N)}(t,x) \right) 
				\right |^p \dd t  \dd x  
				\right)^{\frac{1}{p}}
				\\ & 
				= \left[
				\int_{[0,L]^2} \int_0^T 
				\mE \left |
				f \left(
				\nabla u^{(N)}(t,x)
				\right)
				\right |^p 
				\dd t  \dd x   
				\right]^{\frac{1}{p}}
				\\ & 
				\leq C
                \left(
				\int_0^T
				\sup_{x \in \T^2 }   
				\mE  \left |
				f \left(
				\nabla u^{(N)}(t,x)
				\right)
				\right |^p 
				\dd t 
				\right)^{\frac{1}{p}}
                \\ & 
				\leq C
				\int_0^{t_N} 
				1
				\dd t  \
                +  C
				\int_{t_N}^T
				{N}^{- \frac{2 \alpha}{p+2}}
				\dd t 
				\\ &  
				\leq 
                C
                {t_N}^{\frac{1}{p}}
                +
                C T
				{N}^{- \frac{2 \alpha}{p+2}}
			\end{align*} 
            where we choose $t_N = {N}^{- \frac{2 \alpha p }{p+2}}$, as in \Cref{Conva:Ratealpha}, and thus this verifies the lemma.
		\end{proof}

	\begin{remark} \label{re:vanishnonlin}
		For the stochastic convolution $P_N Z \coloneqq P_N Z^{(0)}$, i.e. for roughness parameter~$\alpha = 0$, in
		\begin{align*} 
			\tilde{u}^{N}(t)
			\coloneqq 
			e^{tA} P_N u_0+\int_0^t e^{(t-s)A}  \mathfrak{F} \left( \tilde u^{N}(s) \right ) \dd s + P_{N} Z(t)
		\end{align*}
		we derived in \cite{BR} the following convergence rate
		\begin{align*}
			\sup_{x \in \T^2 }   
			\left( \mE \left |
			f \left(
			\nabla \tilde{u}^{N}(t,x)
			\right)
			\right |^p \right)^{\frac{1}{p}}
			\lesssim
			\ln \left(N \right)^{- \frac{1}{p+2}}.
		\end{align*}
		By choosing $K_{N} = \ln(N)^\vartheta \vee M$ the proof proceeds as in the proof of \Cref{Conva:Ratealpha}.
	\end{remark}

	\noindent
	\begin{remark}
		The application of a non-increasing taming rate $\eta : \N \to (0, \infty)$ to the orthogonal projection and a non-decreasing function $\zeta : \N \to \N$ such that we have the stochastic convolution
		\begin{align*}
			\eta (N) P_{\zeta(N)} Z (t,x)
			\coloneqq
			\eta(N) \sum_{k\in \ZZ, |k| \leq {\zeta(N)}}\int_0^te^{- (t-s)\mu_k^2}\dd \beta_k(s) e_k(x),
		\end{align*}
		may improve several of the convergence rates discussed in this article. However, a detailed analysis of these potential improvements will be provided in our forthcoming paper \cite{BLR}.
	\end{remark}
	
	\subsection{Noise with Converging Diffusion Coefficients} \label{converg:noise}
	Inspired by the results of \cite{BR:GR26}, where we showed that a sequence of
diffusion coefficients $(\sigma_N)_{N\in\N}$ must not converge too rapidly to
zero, since this would destroy the linearization effect caused by the vanishing
of the nonlinearity, and taking into account the growth rate of the regularized
stochastic convolution $(P_N Z^{(\alpha)})_{N\in\N}$ established in
\Cref{convrate:Zalpha}, we consider a sequence
\[
    (\sigma(N))_{N\in\N}\subset(0,1)
\]
depending only on the spatial truncation parameter $N$.
We require that $\sigma(N)\to0$ as $N\to\infty$, while at the same time
\begin{equation*}
    \E
        \left\|
            \sigma(N)P_N Z^{(\alpha)}(t)
        \right\|_{\mathcal H^1}^2
    \xrightarrow{N\to\infty}\infty.
\end{equation*}
A sufficient condition is that, for some fixed but arbitrary
$\zeta\in(0,\frac{1}{2})$ and all sufficiently large $N\in\N$, we have
\begin{equation*}
    \sigma(N)\geq
    \begin{cases}
        \bigl(\log N\bigr)^{-\zeta},
        & \alpha=0,\\[0.3em]
        N^{- 2 \zeta \alpha },
        & \alpha\in(0,1).
    \end{cases}
\end{equation*}
	Let 
	\begin{align*}
		Z_{\sigma(N), N}^{(\alpha)}  (t,x) 
		\coloneqq 
		\sigma(N) P_{N}  Z^{(\alpha)} (t,x) 
		=
		\sigma(N) \sum_{k\in\ZZ, |k| \leq N} \mu_k^{\frac{\alpha}{2}}\int_0^te^{- (t-s)\mu_k^2}\dd \beta_k(s) e_k(x)
	\end{align*}
	for a null sequence $\sigma : \N \to (0,1)$.
    Correspondingly, we
	denote 
	\begin{align*}
		\tilde{u}_{\alpha, \sigma(N)}^{(N)}(t)
		=
		e^{tA} P_N u_0+\int_0^t e^{(t-s)A} P_N \mathfrak{F} \left( \tilde{u}_{\alpha, \sigma(N)}^{(N)}(s) \right) \dd s +Z_{\sigma(N), N}^{(\alpha)}(t).
	\end{align*}

    Following the same proof as for \Cref{Conva:Ratealpha}, we obtain the result below.
	\begin{lemma}
    Let $N \geq 3$ sufficiently large.
    Let $u_0 \in \mathcal{H}^{1 + \eta} \left(\T^2 \right) \cap W^{1 , \infty} \left(\T^2 \right)$ for $\eta  \in (0, 4)$. 
		For each $p\geq1$ and $t \in (0,T]$ it holds that for $f(z)\coloneqq \frac{z}{1+|z|^2}, z\in\mR^2,$ 
		\begin{align*}
			\sup_{x \in \T^2 }   
			\left( \mE  \left |
			f \left(
            \nabla
			\tilde{u}_{\alpha, \sigma(N)}^{(N)}(t,x)
			\right)
			\right |^p
			\right)^{\frac{1}{p}}
			\lesssim 
			\begin{cases}
				\left[
				\sigma(N)^2
				\ln \left(N \right)
				\right]^{- \frac{1}{p+2}}, & 
				\quad \text{for $\alpha =0$,}
				\\
				\left[
				\sigma(N)^2
				N^{2 \alpha}
				\right]^{- \frac{1}{p+2}}, & 
				\quad \text{for $\alpha \in (0,1)$.}
			\end{cases}
		\end{align*}
	\end{lemma}
	
	\begin{proof}
		By choosing $\vartheta \coloneqq \frac{1}{p+2} \in \left(0, \frac{1}{2} \right)$ and $
		K_{N, \sigma(N)}  \coloneqq \left[ \sigma(N)^2 \ln \left({N}\right) \right]^\vartheta \vee M
		$ if $\alpha=0$ holds, and by choosing
		$
		K_{N, \sigma(N)} \coloneqq \left[ \sigma(N)^2 N^{2\alpha} \right]^\vartheta \vee M
		$ if $\alpha \in (0,1)$ holds,
		the proof is analogous to the proof of \Cref{Conva:Ratealpha} as 
		\begin{align*}
			Z_{\sigma(N), N}^{(\alpha)}  \equiv \sigma(N) P_N  Z_{ }^{(\alpha)} 
		\end{align*}
		holds and $\sigma(N)$ is a constant factor that does not depend on $x \in \T^2$, $t \in (0,T]$ or~$\omega  \in \Omega$.
	\end{proof}
	 
	\noindent
	Now, it is straightforward to obtain the following result:
	\begin{corollary}
		For fixed but arbitrary $\zeta \in \big( 0 , \frac{1}{2} \big)$ according to the roughness of the noise, consider for every $N \in \N$
		\begin{align*}
			\sigma(N) \coloneqq
			\begin{cases}
				\ln (N)^{-\zeta} & \quad \text{ for $\alpha =0$,} \\ 
				N^{- 2 \zeta \alpha  } & \quad \text{ for $\alpha \in (0,1)$}.
			\end{cases}
		\end{align*}
        Then 
        for each $p\geq1$ and each  $t \in (0,T]$ we have the following, for $f(z)\coloneqq \frac{z}{1+|z|^2}, z\in\mR^2,$
		\begin{align*} 
		 \sup_{x \in \T^2 }   
			\left( \mE  \left |
			f \left(
			\nabla \tilde{u}_{\alpha, \sigma(N)}^{(N)}(t,x)
			\right)
			\right |^p \right)^\frac{1}{p}
			\lesssim 
			\begin{cases}
				\ln \left( N \right)^{\frac{2 \zeta - 1}{(p+2)}} \quad & \text{for $\alpha =0$}, \\
				N^{\frac{2 \alpha  \left(2 \zeta - 1 \right)}{(p+2)}} \quad & \text{for $\alpha \in (0,1)$.}
			\end{cases}
		\end{align*} 
	\end{corollary}

		\section{Decomposition of the Error Function}  \label{sec:proof}
        In this section, we will only consider the case $\alpha \in (0,1)$.
		Let
		$
		e^{(N)}_h  =  u^{(N)}-u_h^{(N)}
		$
		be the error function and $Q_N \coloneqq I-P_N$ be an orthogonal projection. 
		Then we have the mild formulation
		\begin{align}
			e^{(N)}_h (t)
			&= e^{tA}Q_N u_0 + \int_0^t e^{(t-s)A}  \left[ \mathfrak{F} (u^{(N)}(s)) - P_N \mathfrak{F} (u_h^{(N)}(\floor{s}_h)) \right]\dd s  \label{e:mild-eh}\\
			&= e^{tA}Q_N u_0 
			+ \int_0^t e^{(t-s)A}  Q_N  \mathfrak{F} (u^{(N)}(s))\dd s 
            \nonumber\\& \qquad
			+\int_0^t e^{(t-s)A}  \left[P_N  \mathfrak{F} (u^{(N)}(s)) - P_N  \mathfrak{F} (u_h^{(N)}(s)) \right]\dd s \nonumber\\& \qquad
			+\int_0^t e^{(t-s)A}  \left[ P_N \mathfrak{F} (u_h^{(N)}(s)) -  P_N \mathfrak{F} (u_h^{(N)}(\floor{s}_h)) \right]\dd s  \nonumber\\
			& \eqqcolon I_1 (t) +I_2(t) + I_3(t) + I_4(t).\nonumber
		\end{align}
        We now bound $I_1$, $I_3$, and $I_4$, while $I_2$ is treated via a Gr\"onwall argument in the next section.

			\begin{lemma}   \label[lemma]{ineq:I1}
				Let $u_0 \in \mathcal{H}^{1 + \eta} \left(\T^2 \right) \cap W^{1 , \infty} \left(\T^2 \right)$ for $\eta  \in (0, 4)$. Then, for each $p \geq 1$, we have
				\begin{align*}
					\left( \mathbb{E}  \sup_{t \in [0,T]} \left \|I_1(t) \right \|_{\mathcal{H}^1} ^p \right)^{\frac{1}{p}} 
					\leq 
					\| u_0\|_{\mathcal{H}^{1+\eta }}
					\left(
					\frac{L}{2 \pi}
					\right)^{\eta } N^{-\eta }.
				\end{align*}
			\end{lemma}
			 
			\begin{proof}
				For $u_0 = \sum_{k \in \ZZ} u_k e_k $  we have
				\begin{align*}
					\left( \mathbb{E}  \sup_{t \in [0,T]} \left \|I_1(t) \right \|_{\mathcal{H}^1} ^p \right)^{\frac{1}{p}}
					& = 
					\sup_{t \in [0,T]}
					\|I_1 (t) \|_{\mathcal{H}^1} 
					\\ & \leq 
					\sup_{t \in [0,T]}
					e^{-t \left( \frac{2\pi}{L}\right)^4} \|Q_N u_0\|_{\mathcal{H}^1}
					\\ &
					\leq 
					\left(
					\sum_{k \in \ZZ  , |k|>N} u_k^2    \mu_k^{1+ \eta }    \mu_k^{- \eta }
					\right)^{\frac{1}{2}}
					\\ &
					\leq 
					\| u_0\|_{\mathcal{H}^{1+\eta }}
					\left(
					\frac{L}{2 \pi}
					\right)^{\eta } N^{-\eta },
				\end{align*}
				where we apply $\mu_k^{-\eta} = \left( \frac{L}{2 \pi |k| } \right)^{2\eta } \leq \left( \frac{L}{2 \pi N } \right)^{2\eta }$ for $|k| > N$. 
			\end{proof}
			
			\begin{lemma}  \label[lemma]{ineq:I2}
				For each $\gamma \in (0,2)$ and $p > \frac{4}{\gamma}$  such that $ \frac{p}{p-1} \in 
				\left( 1 \ , \ \frac{1}{1-\frac{\gamma}{4}} \right)$,  we obtain 
				\begin{align*}
					\left( \mathbb{E} \sup_{t \in [0,T]} \left \|  
					I_2(t) \right \|_{\mathcal{H}^1} ^p \right)^{\frac{1}{p}} 
					\leq C 
					N^{-2 + \gamma - \frac{2 \alpha}{p + 2}}.
				\end{align*}
			\end{lemma}

			\begin{proof}
				Applying Hölder's inequality, Minkowski inequality, and \Cref{cor:vanolin}, we conclude
				\begin{align*}
					& \left( \mathbb{E}  \sup_{t \in [0,T]} \left \|I_2(t) \right \|_{\mathcal{H}^1} ^p \right)^{\frac{1}{p}}
					\\ & =
					\left( \mathbb{E}  \sup_{t \in [0,T]}
					\left \|  \int_0^t e^{(t-s)A}  Q_N  \left[  \mathfrak{F} \left(u^{(N)}(s) \right) \right] \dd s \right \|_{\mathcal{H}^1}
					^p \right)^{\frac{1}{p}}
					\\ &
					\leq 
                    \left \| Q_N \right \|_{L \left( \mathcal{H}^{-1} ,\mathcal{H}^{-3+\gamma} \right)} 
                    \left( \mathbb{E}    \sup_{t \in [0,T]} \left( \int_0^t \left \| e^{(t-s)A} \nabla \cdot \right \|_{L \left( \mathcal{H}^{-2+\gamma} \left(\T^2 , \R^2 \right) , \mathcal{H}^{1} \right)} 
					\left \| f \left( \nabla u^{(N)}(s) \right) \right \|_{L^2} \dd s \right)^p \right)^{\frac{1}{p}}  
					\\
					& \leq 
                    \left \| Q_N \right \|_{L \left( \mathcal{H}^{-1} ,\mathcal{H}^{-3+\gamma} \right)} 
                    \left( \mathbb{E}    \sup_{t \in [0,T]} \left( \int_0^t \left \| e^{(t-s)A} \right \|_{L \left( \mathcal{H}^{-3+\gamma} , \mathcal{H}^{1} \right)} 
					\left \| f \left( \nabla u^{(N)}(s) \right) \right \|_{L^2} \dd s \right)^p \right)^{\frac{1}{p}}  
					\\
						& \leq 
						C_{\gamma}
						\| Q_N \|_{L \left( \mathcal{H}^{-1} ,  \mathcal{H}^{-3 + \gamma} \right)}   
					\left(
					\int_0^T 
					 s^{\left(-1+\frac{ \gamma}{4}\right)\frac{p}{p-1}}  \dd s  \right)^{\frac{p-1}{p}}
					\left( \mathbb{E}
					\int_0^T
					\left \| f \left( \nabla u^{(N)}(s) \right) \right \|_{L^2}^{p} \dd s \right)^{\frac{1}{p}} 
					\\
						& \leq
						C_{\gamma}
					N^{-2 + \gamma}   T^{(-1 + \frac{\gamma}{4}) +\frac{p-1}{p}}
					\int_0^T
					\left( \mathbb{E} 
					\left \| f \left( \nabla u^{(N)}(s) \right) \right \|_{L^2}^{ p}  \right)^{\frac{1}{p}} 
					\dd s 
					\\
					& \leq 
					C_T
					N^{-2 + \gamma}    N^{- \frac{2 \alpha}{p + 2}}.
				\end{align*}
			This shows the assertion.
				\end{proof}

            \begin{lemma}  \label[lemma]{ineq:I3}
				For $p >1$  and each $r \in [0,T]$ we have
				\begin{align*}
					 \left(\mathbb{E} 
					\sup_{t \in [0,r]}
					\left \|
					I_3(t)
					\right \|_{\mathcal{H}^1}^p 
					\right)^{\frac{1}{p}}    
					& \leq 
					\frac{1}{\sqrt{2e}} \int_0^r (r-s)^{-\frac{1}{2}}\left( \mathbb{E}   \sup_{\tau \in [0,s]}
				\|e^{(N)}_h (\tau ) \|_{\mathcal{H}^1}^p
				\right)^{\frac{1}{p}}  \dd s .
				\end{align*}
			\end{lemma}
			
			\begin{proof}
            First of all, in \cite[Section~4]{BR} we showed that $\mathfrak{F} : \mathcal{H}^1 \to \mathcal{H}^{-1}$ is Lipschitz continuous.
            Thus, by the Minkowski inequality we derive
				\begin{align*}
					  \left( \mathbb{E} 
					\sup_{t \in [0,r]}
					\left \|
					I_3(t)
					\right \|_{\mathcal{H}^1}^p 
					\right )^{\frac{1}{p}}  
					& =
                    \left( \mathbb{E}
					\sup_{t \in [0,r]}
					\left \|
					\int_0^t e^{(t-s)A}  P_N  \left[ \mathfrak{F} (u^{(N)}(s)) -   \mathfrak{F} (u_h^{(N)}(s)) \right]\dd s
					\right \|_{\mathcal{H}^1}^p 
					\right )^{\frac{1}{p}}   
                    \\
                    & \leq 
					\int_0^r
                    \left \| e^{(r-s)A}  \right \|_{L \left(\mathcal{H}^{-1} , \mathcal{H}^1 \right)}
                    \left(\E  \sup_{\tau \in [0,s]} \left \|\mathfrak{F} (u^{(N)}(\tau)) -   \mathfrak{F} (u_h^{(N)}(\tau))
					\right\|^p_{\mathcal{H}^{-1}} \right)^{\frac{1}{p}} \dd s
                    \\
                    & \leq 
                    \frac{1}{\sqrt{2e}}
					\int_0^r (r-s)^{-\frac{1}{2}}   
                    \left(\E   \sup_{\tau \in [0,s]} \left \| u^{(N)}(\tau) - u_h^{(N)}( \tau )
					\right\|^p_{\mathcal{H}^1} \right)^{\frac{1}{p}} 
					\dd s 
					.
				\end{align*}
			This confirms the claim.
			\end{proof}
			
			\begin{lemma}  \label[lemma]{ineqaux:I4}
			Let  $\alpha \in (0,1)$ and $u_0 \in \mathcal{H}^{1 + \eta} \left(\T^2 \right) \cap W^{1 , \infty} \left(\T^2 \right)$ for $\eta  \in (0, 4)$   and $\gamma \in (0,2)$. 
			For $p >1$, such that~$\frac{p}{p-1} \in \big(1 , \frac{1}{1-\frac{\gamma}{4}} \big)$, $\varepsilon \in \big( 0 , \frac{1}{2} - \frac{1}{p} \big)$, and  there exists a constant $C = C\left( \alpha , p , \sigma, L, T, \left \| u_0 \right \|_{W^{1 , \infty}} \right) >0$ we obtain 
				\begin{align*} 
					& \left( \mathbb{E}
					\sup_{t \in [0,T]} \left \| u_h^{(N)}(t) - u_h^{(N)}(\floor{t}_h)
					\right\|^p_{\mathcal{H}^1} \right)^{\frac{1}{p}} 
					\\
					& \leq C \left(
					h^{\frac{\eta }{4}}
					\left \|
					u_0
					\right \|_{\mathcal{H}^{1+ \eta }}
					+
					N^{2  + \alpha }
					h^{\frac{1}{2} - \frac{1}{p}  - \varepsilon}
					+
					h^{\frac{1}{2}- \frac{\gamma}{4}} \left[ N^{- \frac{2 \alpha}{p +2}} + h \right] 
					+
					{h}^{\frac{1}{2} - \frac{1}{p}}
					\left[                       N^{- \frac{2 \alpha}{p +2}}                      +   					h                      \right] 
					\right).
				\end{align*}
			\end{lemma}

			\begin{proof}
				According to \eqref{ineq:SGalpha} we have for each $\beta \in(0,1]$
				\[
				\left \| (e^{tA} - \operatorname{Id})u \right \|_{\mathcal{H}^1}
				\leq C t^{\beta} \left\|(-A)^{\beta} u\right \|_{\mathcal{H}^{1}}
				\leq C t^{\beta} \left\|u\right \|_{\mathcal{H}^{4\beta+1}}
				\]
				for each $t \in [0,T]$ (see \eqref{norm:H}).
				Therefore, for $\gamma \in (0,2)$  we obtain
				\[
				\left \| (e^{tA} - \operatorname{Id}) \right \|_{L(\mathcal{H}^{3-\gamma},\mathcal{H}^1)}
				\leq C t^{\frac{1}{2} - \frac{\gamma}{4}}.
				\]
				
				\noindent
				We now decompose the difference as follows:
				\begin{align*}
					u_h^{(N)}(t) - u_h^{(N)}(\floor{t}_h)
					& =
					\left(
					e^{t A } - e^{\floor{t}_h A }
					\right) P_N
					u_0
					+
					\sigma P_N \left(Z^{(\alpha)}(t) - Z^{(\alpha)}\left ( \floor{t}_h \right) \right)
					\\ & \quad +
					\int_0^{\floor{t}_h} 
					\left (
					e^{(t-s)A} - e^{\left(\floor{t}_h-s \right)A}
					\right ) P_N 
					\mathfrak{F} \left(  
					u_h^{(N)}(\floor{s}_h)
					\right)
					\dd s 
					\\ & \quad +
					\int_{\floor{t}_h}^t e^{(t-s)A}   
					P_N
					\mathfrak{F} \left(  
					 u_h^{(N)}(\floor{s}_h) \right)
					\dd s 
					\\ &
					\eqqcolon
					\hat{I}_{41}(t) + 
					\hat{I}_{42} (t) + 
					\hat{I}_{43} (t) + 
					\hat{I}_{44}(t).
				\end{align*}
				
				\noindent
                Applying \eqref{ineq:SGalpha}, with $\frac{\eta}{4} \in (0,1]$, we obtain
				\begin{align}  \label {ineq:I41}
					\left( \mathbb{E}
					\sup_{t \in [0,T]}
					\left \|
					\hat{I}_{41} (t)
					\right \|_{\mathcal{H}^1} ^p \right)^{\frac{1}{p}} 
					& =
					\sup_{t \in [0,T]}
					\left \|
					\hat{I}_{41} (t)
					\right \|_{\mathcal{H}^1} 
					\\
					&
					\leq
					\sup_{t \in [0,T]}
					\left \|
					\left(
					e^{t A } - e^{\floor{t}_h A }
					\right) P_N
					u_0
					\right \|_{\mathcal{H}^1}          \nonumber
					\\
					&
					\leq
					\sup_{t \in [0,T]}
					\left \|
					e^{ \floor{t}_h A }
					\right\|_{L \left(\mathcal{H}^{1} ,
						\mathcal{H}^1  
						\right)}
					\left \|
					 \left( e^{(t- \floor{t}_h) A } - \operatorname{Id} \right)  
					 P_N
					u_0
					\right \|_{\mathcal{H}^{1 }}          \nonumber
					\\
					&
						\leq 
					C \sup_{t \in [0,T]}
					\left | t-\floor{t}_h \right|^{\frac{\eta }{4}}
					\left \|
					\left(
					\left(-\Delta \right)^2 \right)^{\frac{\eta}{4}}
					u_0
					\right \|_{\mathcal{H}^{1}} \nonumber
					\\ & \leq C h^{\frac{\eta }{4}}
					\left \|
					u_0
					\right \|_{\mathcal{H}^{1+ \eta }}. \nonumber
				\end{align} 
				Applying \Cref{lem:Z-beta}, we derive 
				\begin{align*} 
					\left( \mathbb{E}
					\left \|
					P_N \left(Z^{(\alpha)}(t) - Z^{(\alpha)} \left( s \right) \right) 
					\right \|_{\mathcal{H}^1}^p \right)^{\frac1p}
					& \leq 
					C
					N^{\alpha   + 2 \psi} |t-s|^{\frac{\psi}{2}}
 				\end{align*}
 				for each $t,s \in [0,T]$,  and $\psi \in (0,1]$.
				Since $Z$ is Gaussian, Kolmogorov's continuity theorem (see \cite[Theorem~2.2.3]{Oks:2000})
				 yields for a $\zeta \in \big(0 , \frac{\psi}{2} - \frac{1}{p} \big)$ 
				that 
				\begin{align*}
					\left( \mE 
					\sup_{t \in [0,T], t \neq \floor{t}_h} 
					\frac{\left \|
						P_N \left(Z^{(\alpha)}(t) - Z^{(\alpha)} \left( \floor{t}_h \right) \right) 
						\right \|_{\mathcal{H}^1}^p}{\left|t-\floor{t}_h \right|^{p \zeta}}
					\right)^{\frac{1}{p}}
					< C N^{\alpha   + 2 \psi}.
				\end{align*}
			  By choosing $\psi =1$, we derive the bound
				\begin{align}  \label {ineq:I42}
					\left( \mE 
					\sup_{t \in [0,T]}
					\left \|
					\hat{I}_{42} (t)
					\right \|_{\mathcal{H}^1}^p \right)^{\frac{1}{p}}
					& = \sigma
					\left( \mE 
					\sup_{t \in [0,T]}
					\left \|
					P_N \left(Z^{(\alpha)}(t) - Z^{(\alpha)}\left ( \floor{t}_h \right) \right) 
					\right \|_{\mathcal{H}^1}^p \right)^{\frac{1}{p}} 
					\\
					& \leq 
					C \sigma N^{2 + \alpha  }\sup_{t \in [0,T]}
					\left|t-\floor{t}_h \right|^{\zeta} \nonumber
					\\
					& \leq 
					C \sigma N^{2 + \alpha    }
					h^{\frac{1}{2} - \frac{1}{p} -\varepsilon}. \nonumber
				\end{align}
                
				\noindent
				From \eqref{ineq:SGalpha} we get
				\begin{align*}
					\left \|
					\left(
					e^{h A} - \operatorname{Id}
					\right) y
					\right \|_{\mathcal{H}^1}
					& \leq
					C h^{\frac{1}{2}- \frac{\gamma}{4}} \left \| \left(- A \right)^{\frac{1}{2}- \frac{\gamma}{4}} y \right \|_{\mathcal{H}^{1}}
					\leq C h^{\frac{1}{2}- \frac{\gamma}{4}} \left \| y \right \|_{\mathcal{H}^{3- \gamma}}
				\end{align*}
				for $y \in \mathcal{H}^{3- \gamma}$ with $\gamma \in (0,2)$.
				Thus, by applying \Cref{Conva:Ratealpha} for the case $s \geq t_N + h$, the Minkowski inequality, and Hölder's inequality, as in the proof of \Cref{ineq:I2}, we derive	
				\begin{align}  \label {ineq:I43}
					& \left( \mathbb{E} 
					\sup_{t \in [0,T]}
					\left \|
					\hat{I}_{43} (t)
					\right \|_{\mathcal{H}^1} ^p \right)^{\frac{1}{p}}
					\\ & \leq 
					\left( \mathbb{E}
					\sup_{t \in [0,T]}
					\left(
					\int_0^{\floor{t}_h} 
					\left \|
					e^{(t-s)A} - e^{\left(\floor{t}_h-s \right)A}
					\right \|_{L\left(
						\mathcal{H}^{-1}, \mathcal{H}^1
						\right)}
					\left \|
					f \left( \nabla u_h^{(N)}(\floor{s}_h)
					\right)
					\right \|_{L^2} \dd s \right)^p
					\right )^{\frac{1}{p}} \nonumber
					\\
					& 
					\leq 
					\sup_{t \in [0,T]}
					\left \|
					e^{\left(t- \floor{t}_h \right)A} - \operatorname{Id}
					\right \|_{L\left(
						\mathcal{H}^{3- \gamma}, \mathcal{H}^1
						\right)} \nonumber
					\\
					& \quad \times
					\left(
					\int_0^{\floor{T}_h}
					\left\| e^{\left(\floor{T}_h-s \right)A}
					\right\|_{L\left(\mathcal{H}^{-1} , 
						\mathcal{H}^{3- \gamma}
						\right)}^{\frac{p}{p-1}}  \dd s  \right)^{\frac{p-1}{p}} 
					\left( \mathbb{E}
					\int_0^T
					\left \| f \left( \nabla u_h^{(N)} \left(\floor{s}_h \right) \right) \right \|_{L^2}^{ p} \dd s  \right)^{\frac{1}{p}} 
					 \nonumber
					\\
					& 
					\leq 
					C h^{\frac{1}{2}- \frac{\gamma}{4}}
					\left(
					\int_0^{\floor{T}_h} 
					s^{\left(-1+\frac{ \gamma}{4}\right)\frac{p}{p-1}}  \dd s  \right)^{\frac{p-1}{p}} \nonumber
					\int_0^T
					\left( \mathbb{E}
					\left \| f \left( \nabla u_h^{(N)} \left(\floor{s}_h \right) \right) \right \|_{L^2}^{ p}  \right)^{\frac{1}{p}} 
					\dd s  \nonumber
						\\
					& 
					\leq C 
					h^{\frac{1}{2}- \frac{\gamma}{4}}   
                    \left[ 
                    N^{- \frac{2\alpha}{p +2}} 
                    +
                     \int_0^{t_N + h}
					\left( \mathbb{E}
					\left \| f \left( \nabla u_h^{(N)} \left(\floor{s}_h \right) \right) \right \|_{L^2}^{ p}  \right)^{\frac{1}{p}} 
					\dd s
                     \right]
                     \nonumber
                     \\
                     & 
					\leq C h^{\frac{1}{2}- \frac{\gamma}{4}}   
                    \left[  
                    N^{- \frac{2 \alpha}{p +2}} 
                    +  
					h 
                    \right]
                    . \nonumber 			
				\end{align}

				\noindent
				Furthermore, by the Minkowski inequality, Hölder's inequality,
                and the same splitting argument as in \eqref{ineq:I43}, we obtain 
				\begin{align}  \label {ineq:I44}
					  \left( \mathbb{E}
					\sup_{t \in [0,T]}
					\left \|
					\hat{I}_{44} (t)
					\right \|_{\mathcal{H}^1} ^p \right)^{\frac{1}{p}}
					  &  \leq 
					\left( \mathbb{E} \sup_{t \in [0,T]} \left(
					\int_{\floor{t}_h}^t \frac{(t-s)^{-\frac{1}{2}} }{\sqrt{2e}}   
					\left \|
					f \left( \nabla
					u_h^{(N)}(\floor{s}_h) \right)
					\right \|_{L^2} 
					\dd s  \right)^{p}	\right )^{\frac{1}{p}} 
					\\ & 
					\leq 
					\left( \mathbb{E} \sup_{t \in [0,T]} 
					\left(
					\int_{\floor{t}_h}^t 
					(t-s)^{- \frac{p}{2(p-1)}}  \dd s  \right)^{p-1}
					\int_{\floor{t}_h}^t 
					\left \| f \left( \nabla u_h^{(N)}(\floor{s}_h) \right) \right \|_{L^2}^{p} \dd s  
					\right)^{\frac{1}{p}} 
						\nonumber 
					\\ & 
					\leq C_p
					\left( \mathbb{E} \sup_{t \in [0,T]}
					\left( t- 
					{\floor{t}_h}    \right)^{\frac{p-2}{2 }}
					\int_{0}^T 
					\left \| f \left( \nabla u_h^{(N)}(\floor{s}_h) \right) \right \|_{L^2}^{p} \dd s  
					\right)^{\frac{1}{p}} 
						\nonumber 
					\\ & 
					\leq C_p
					{h}^{\frac{1}{2} - \frac{1}{p}}
					\int_{0}^T 
					\left( \mathbb{E}
					\left \| f \left( \nabla u_h^{(N)}(\floor{s}_h) \right) \right \|_{L^2}^{p}  
					\right)^{\frac{1}{p}} \dd s 
					\nonumber  
					\\ &
					\leq C_{p,T} {h}^{\frac{1}{2} - \frac{1}{p}} 
                    \left[  
                    N^{- \frac{2 \alpha}{p +2}} 
                    +  
					h 
                    \right]
                     . \nonumber
				\end{align}
							
				\noindent
				Finally, applying the triangle inequality and combining \eqref{ineq:I41}, \eqref{ineq:I42}, \eqref{ineq:I43}, and \eqref{ineq:I44}, we conclude 
				\begin{align*} 
					& \left (\E   \sup_{t \in [0,T]} \left \| u_h^{(N)}(t) - u_h^{(N)}(\floor{t}_h)
					\right\|^p_{\mathcal{H}^1} \right)^{\frac{1}{p}} 
					\\ & 
					\leq C_{p,T} \left(
					h^{\frac{\eta }{4}}
					\left \|
					u_0
					\right \|_{\mathcal{H}^{1+ \eta }}
					+
				\sigma N^{2  + \alpha } h^{\frac{1}{2} - \frac{1}{p}  - \varepsilon}
					+ 
                    h^{\frac{1}{2}- \frac{\gamma}{4}} \left[  
                    N^{- \frac{2 \alpha}{p +2}} 
                    +  
					h 
                    \right] 
                    +
					{h}^{\frac{1}{2} - \frac{1}{p}}\left[  
                    N^{- \frac{2 \alpha}{p +2}} 
                    +  
					h 
                    \right]   
					\right).
				\end{align*}
			This establishes the claim.
			\end{proof}
			
			\begin{lemma}  \label[lemma]{ineq:I4}
				Under the assumptions of \Cref{ineqaux:I4},  we have
				\begin{align*}
					\left( \mathbb{E}
					\sup_{t \in [0,T]}
					\left \|
					I_4(t)
					\right \|_{\mathcal{H}^1}^p 
					\right)^{\frac{1}{p}}    
					& \leq 
					C 
					\sqrt{T}
					\left(
					h^{\frac{\eta }{4}}
					\left \|
					u_0
					\right \|_{\mathcal{H}^{1+ \eta }}
					+
					  N^{2  + \alpha }
					h^{\frac{1}{2} - \frac{1}{p}  - \varepsilon}
					+
					\left[  
                    N^{- \frac{2 \alpha}{p +2}} 
                    +  
					h 
                    \right] 
					\left(
					h^{\frac{1}{2}- \frac{\gamma}{4}} 
					+
					{h}^{\frac{1}{2} - \frac{1}{p}}
					\right)
					\right).
				\end{align*}
			\end{lemma}
			
			\begin{proof}
				By \Cref{ineqaux:I4} it holds that
				\begin{align*}
					  \left( \mathbb{E}
					\sup_{t \in [0,T]}
					\left \|
					I_4(t)
					\right \|_{\mathcal{H}^1}^p 
					\right)^{\frac{1}{p}}  
					& \leq 
					\int_0^T (T-s)^{-\frac{1}{2}}   \left( \mathbb{E} \sup_{t \in [0,T]} \left \| u_h^{(N)}(t) - u_h^{(N)}(\floor{t}_h)
					\right\|^p_{\mathcal{H}^1} \right)^{\frac{1}{p}} 
					\dd s 
					\\
					& \leq   C_p
					\sqrt{T}
					\left(
					h^{\frac{\eta }{4}}
					\left \|
					u_0
					\right \|_{\mathcal{H}^{1+ \eta }}
					+
					\sigma  N^{2  + \alpha }
					 h^{\frac{1}{2} - \frac{1}{p} - \varepsilon } 
					+
					\left[  
                    N^{- \frac{2 \alpha}{p +2}} 
                    +  
					h 
                    \right] 
					\left(
					h^{\frac{1}{2}- \frac{\gamma}{4}} 
					+
					{h}^{\frac{1}{2} - \frac{1}{p}}
					\right)
					\right).
				\end{align*}
			This confirms the claim.
			\end{proof}

		
		\section{Proof of the Main Result}  \label {sec:mainproof}

		Finally, we can make a proof of the main result of this chapter \Cref{thm:mainEuler}.
		\begin{proof}[Proof of \Cref{thm:mainEuler}]
			We define the auxiliary function  
			$$
			k(s)  \coloneqq 
			\left(\mathbb{E} 
			\sup_{t \in [0,s]}
			\left \| 
			e^{(N)}_h (t) 
			\right \|_{\mathcal{H}^1}^p 
			\right )^{\frac{1}{p}}
            , \quad
            \text{ for   $s \in [0,T]$}
            $$  
            and the constant
			\begin{align*}
				a  &
				\coloneqq C \left( 
				\| u_0\|_{\mathcal{H}^{1+\eta }}
				N^{-\eta }
				+ 
				N^{-2 + \gamma - \frac{2 \alpha}{p + 2}}
				\right) 
                \\ 
                & \qquad
				+ C
				\left(
				h^{\frac{\eta }{4}}
				\left \|
				u_0
				\right \|_{\mathcal{H}^{1+ \eta }}
				+
				   N^{2  + \alpha }
				 h^{\frac{1}{2} - \frac{1}{p} - \varepsilon } 
				+
				\left[ N^{- \frac{2 \alpha}{p +2}}  + h  \right] 
				\left(
				h^{\frac{1}{2}- \frac{\gamma}{4}} 
				+
				{h}^{\frac{1}{2} - \frac{1}{p}}
				\right)
				\right).
			\end{align*}
			By combining the results of \Cref{ineq:I1}, \Cref{ineq:I2}, \Cref{ineq:I3}, and \Cref{ineq:I4}, 
            we obtain for each~$r \in [0, T]$
			\begin{align*} 
				k(r) & =
				\left( \mathbb{E}
				\sup_{t \in [0,r]}
				\left \| 
				e^{(N)}_h (t) 
				\right \|_{\mathcal{H}^1}^p 
				\right )^{\frac{1}{p}}  \\
				& \leq 
				\left( \mathbb{E}
				\sup_{t \in [0,r]}
				\left \|
				I_1(t)
				\right \|_{\mathcal{H}^1}^p 
				\right )^{\frac{1}{p}}  
				+
				\left( \mathbb{E}
				\sup_{t \in [0,r]}
				\left \|
				I_2(t)
				\right \|_{\mathcal{H}^1}^p 
				\right)^{\frac{1}{p}}  
				+
				\left( \mathbb{E}
				\sup_{t \in [0,r]}
				\left \|
				I_4(t)
				\right \|_{\mathcal{H}^1}^p 
				\right )^{\frac{1}{p}}    
				+ \left( \mathbb{E}
				\sup_{t \in [0,r]}
				\left \|
				I_3(t)
				\right \|_{\mathcal{H}^1}^p 
				\right)^{\frac{1}{p}}  
				\\ 
				& 
				\leq C \left( 
				\| u_0\|_{\mathcal{H}^{1+\eta }}
				N^{-\eta }
				+ 
				N^{-2 + \gamma - \frac{2 \alpha}{p + 2}}
				\right)
				\\ & \quad 
				+ C
				\left(
				h^{\frac{\eta }{4}}
				\left \|
				u_0
				\right \|_{\mathcal{H}^{1+ \eta }}
				+
				   N^{2  + \alpha }
				 h^{\frac{1}{2} - \frac{1}{p} - \varepsilon } 
				+
				\left[ N^{- \frac{2 \alpha}{p +2}}  + h  \right] 
				\left(
				h^{\frac{1}{2}- \frac{\gamma}{4}} 
				+
				{h}^{\frac{1}{2} - \frac{1}{p}}
				\right)
				\right)
				\\ & \quad  +
				\frac{1}{\sqrt{2e}} \int_0^r (r-s)^{-\frac{1}{2}} \left( \mathbb{E}   \sup_{\tau \in [0,s]}
				\|e^{(N)}_h (\tau ) \|_{\mathcal{H}^1}^p
				\right)^{\frac{1}{p}}  \dd s 
				\\ & 
				=
				a + \frac{1}{\sqrt{2e}} \int_0^r (r-s)^{-\frac{1}{2}} k(s)  \dd s .
			\end{align*}
			 
			\noindent
            Note that by the definition of $k$ \eqref{e:mild-eh}, and the boundedness of $f$ it is easy to check that $k$ is finite and bounded. Thus the previous inequality allows us, by applying the Grönwall result from \Cref{lem:G-HDG}, to show that there is a constant 
            $C = C\left(\alpha , \eta , p , \gamma,  \varepsilon, \sigma, L, T, \left \| u_0 \right \|_{W^{1 , \infty}} \right) >0$ such that
			\begin{align*}
				k(r) 
				& =
				\left( \mathbb{E}
				\sup_{t \in [0,r]}
				\left \| 
				e^{(N)}_h (t) 
				\right \|_{\mathcal{H}^1}^p 
				\right)^{\frac{1}{p}} 
				\\
				& \leq 
                a  E_{\frac{1}{2}}   \left( \frac{\pi}{2e} T \right) 
				\\
				& \leq C \left( 
				\| u_0\|_{\mathcal{H}^{1+\eta }}
				N^{-\eta }
				+ 
				N^{-2 + \gamma - \frac{2 \alpha}{p + 2}}
				\right)
				\\ & \quad 
				+ C
				\left(
				h^{\frac{\eta }{4}}
				\left \|
				u_0
				\right \|_{\mathcal{H}^{1+ \eta }}
				+
				   N^{2  + \alpha }
				 h^{\frac{1}{2} - \frac{1}{p} - \varepsilon } 
				+
				\left[ N^{- \frac{2 \alpha}{p +2}}  + h  \right] 
				\left(
				h^{\frac{1}{2}- \frac{\gamma}{4}} 
				+
				{h}^{\frac{1}{2} - \frac{1}{p}}
				\right)
				\right)  
				\\
				& \leq  C \left( 
				\| u_0\|_{\mathcal{H}^{1+\eta }} \left( h^{\frac{\eta }{4}} +
				N^{-\eta }
				\right)
				+
				   N^{2  + \alpha }
				 h^{\frac{1}{2} - \frac{1}{p} - \varepsilon } 
				+ 
				N^{-2 + \gamma - \frac{2 \alpha}{p + 2}}
				+
				{h}^{\frac{1}{2} - \frac{\gamma}{4}}
				\left[ N^{- \frac{2 \alpha}{p +2}}  + h  \right] 
				\right).
			\end{align*}
		This establishes the desired result.
		\end{proof}
		
\section{Numerical Simulations} \label{num:simula}

Simulations are conducted on the domain $[0,1]^2$ up to time $T = 10$ using a step size of $h = 0.001$, $ n = T/h = 10^4$ steps, and parameters $\sigma = 0.11$ and $\delta = 0.02$.
    This differs from the normalisation $\delta=1$ used in the
theoretical analysis in the previous sections. The smaller value is chosen in order to
make the linearly unstable modes visible. Indeed, linearization
around a spatially homogeneous state yields the modal growth rate
$
   -\delta\mu_k^2 + \mu_k
$,
so that modes satisfying $0<\mu_k<\delta^{-1}$ are linearly
unstable.
We define $\lambda_k \coloneqq  -\delta \mu_k^2$ for each $k \in \ZZ$.
    The initial condition $u_0$ is set to the constant value $0$. The values presented in the tables below are sample means of the displayed norms over 100 simulations at time $T = 10$.
For each nonzero Fourier mode $k \in \ZZ$, the exponential Euler scheme from \Cref{def:exponential-Euler} takes the form
\begin{align*}
    \widehat{v}^{(N)}_{j+1,k}
    &=
    e^{-\delta h\mu_k^2}\widehat{v}^{(N)}_{j,k}
    +
    \frac{1-e^{-\delta h\mu_k^2}}
         {\delta\mu_k^2}
    \widehat{P_N\mathfrak{F}\bigl(v^{(N)}_j\bigr)}_k
    +
    \sigma\widehat{P_N\widetilde{Z}_j}_k,
\\
    \widehat{P_N\widetilde{Z}_j}_k
    & =
    \mu_k^{\frac{\alpha}{2}}
    \sqrt{
        \frac{1-e^{-2\delta h\mu_k^2}}
             {2\delta\mu_k^2}
    }\,
    \xi_{j,k},
\end{align*}
where  $\left\{ \xi_{j,k} \,  : \,  j = 0 , \hdots , n - 1 \,  ; \, k \in \ZZ \right \} $ is a family of independent standard Gaussian random variables.
The same realisations of $\xi_{j,k}$ are used in the recursions for
$v^{(N)}$ and $Z^{(N)}$.
Similarly, the exponential Euler scheme for the stochastic convolution
is given by
\begin{align*}
    \widehat{Z}^{(N)}_{j+1,k}
    &=
    e^{-\delta h\mu_k^2}\widehat{Z}^{(N)}_{j,k}
    +
    \sigma\widehat{P_N\widetilde{Z}_j}_k,
    \qquad
    \widehat{Z}^{(N)}_{0,k}=0.
\end{align*} 
The difference, defined by
\[
    p_h^{(N)}(jh)
    \coloneqq
    v_j^{(N)}-Z_j^{(N)},
\]
represents the nonlinear remainder.

	\noindent
	\subsection{Result for Roughness Parameter $\alpha = 0.4$} For $\alpha = 0.4$, the $\infty$-norm of the gradient of $v^{(N)}$ does not exhibit systematic growth over the tested resolutions.

	\begin{table}[H]
		\centering
		\begin{tabular}{lrrr}
			\toprule
			Number of Fourier modes & $\left\| v^{(N)} \right \|_{\infty}$ & $\left\| \nabla  v^{(N)} \right \|_{\infty}$  & $\left\| v^{(N)} \right \|_{C^1}$ \\
			\midrule
			${(2^{2})}^2 =16$ & 0.130713 & 0.614351 & 0.745064 \\
			${(2^{3})}^2 =64$ & 0.122107 & 0.670314 & 0.792421 \\
			${(2^{4})}^2 =256$ & 0.133159 & 0.746035 & 0.879194 \\
			${(2^{5})}^2 =1024$ & 0.139909 & 0.760671 & 0.900580 \\
			${(2^{6})}^2 =4096$ & 0.141632 & 0.761279 & 0.902910 \\
			${(2^{7})}^2 =16384$ & 0.143825 & 0.760267 & 0.904092 \\
			${(2^{8})}^2 =65536$ & 0.144930 & 0.751583 & 0.896513 \\
			${(2^{9})}^2 =262144$ & 0.146088 & 0.738105 & 0.884193 \\
			${(2^{10})}^2 =1048576$ & 0.147054 & 0.721802 & 0.868856 \\
			\bottomrule
		\end{tabular}
			\caption[Norms of the exponential Euler scheme in Fourier space for $\alpha = 0.4$.]{Norms of the exponential Euler scheme in Fourier space with respect to the number of Fourier modes for $\alpha = 0.4$.}
		\label{tabelle:1}
	\end{table}
	
	\noindent
	Here, no decay of $v^{(N)}-Z^{(N)}$ is visible over the tested resolutions. More precisely, the simulated data do not resolve the expected slow decay of the nonlinear contribution. In the simulation of $v^{(N)}$, the noise-dominated structure of $Z^{(N)}$ is masked by the characteristic hill growth shown in the figure below.
	\begin{figure}[H]
		\begin{center}
			\includegraphics[scale = 0.46]{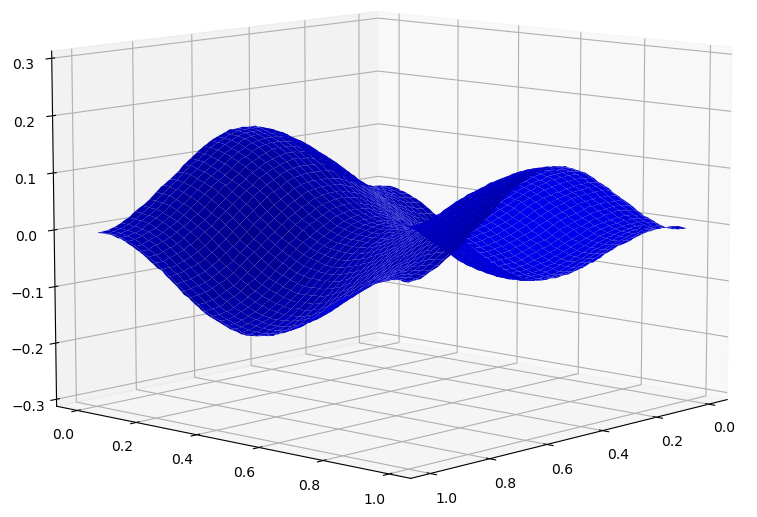}
			\caption[Simulation of $v^{(N)}$ for $\alpha = 0.4$.]{Simulation of $v^{(N)}$ at time $T=10$ for $\alpha = 0.4$ and $N = 2^7$.  } 
		\end{center}
	\end{figure}
	 
	\noindent
	
	\noindent
	The table below shows that, in the numerical simulations, the error $\|\nabla p^{(N)}_h\|_\infty$ stabilises at approximately $0.45$; thus, no decay is visible over the tested resolutions, we would need a much larger $N$.
		\begin{table}[H]
		\centering
		\begin{tabular}{lrrr}
			\toprule
			Number of Fourier modes &  $\left\| p^{(N)}_h \right \|_{\infty}$ &  $\left\| \nabla p^{(N)}_h \right \|_{\infty}$ &  $\left\|  p^{(N)}_h \right \|_{C^1}$ \\
			\midrule
			${(2^{2})}^2 =16$ & 0.102194 & 0.436299 & 0.538493 \\
			${(2^{3})}^2 =64$ & 0.108146 & 0.464480 & 0.572626 \\
			${(2^{4})}^2 =256$ & 0.124719 & 0.474413 & 0.599132 \\
			${(2^{5})}^2 =1024$ & 0.135589 & 0.459246 & 0.594835 \\
			${(2^{6})}^2 =4096$ & 0.139585 & 0.451026 & 0.590611 \\
			${(2^{7})}^2 =16384$ & 0.142452 & 0.452216 & 0.594668 \\
			${(2^{8})}^2 =65536$ & 0.144310 & 0.453707 & 0.598017 \\
			${(2^{9})}^2 =262144$ & 0.145731 & 0.456384 & 0.602115 \\
			${(2^{10})}^2 =1048576$ & 0.146847 & 0.458761 & 0.605608 \\
			\bottomrule
		\end{tabular}
		\caption[Norms of the nonlinear remainder for $\alpha = 0.4$.]{Norms of the nonlinear remainder with respect to the number of Fourier modes with $\alpha = 0.4$.}
	\label{tabelle:2}
	\end{table}
\noindent
\Cref{tabelle:2} summarizes the numerical values of the norms of the nonlinear remainder 
for $\alpha = 0.4$ and increasing numbers of Fourier modes. 
The corresponding dependence on the Fourier resolution is illustrated in \Cref{figure:line}.
	\begin{figure}[H]
		\centering
		\includegraphics[width=0.7\textwidth]{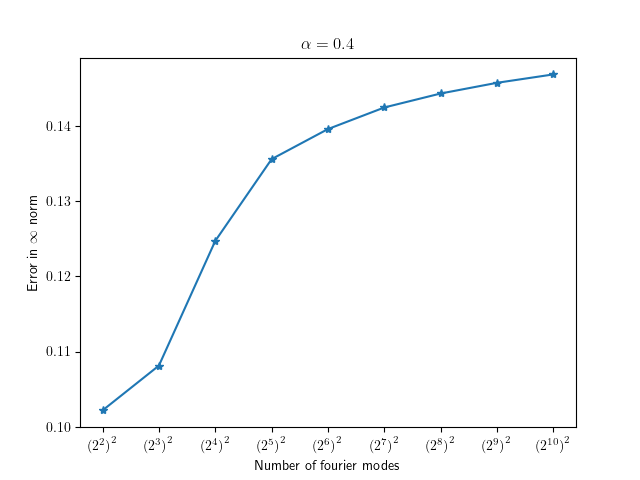} \\
		\caption[Simulation of $\big\| p^{(N)}_h \big \|_{\infty}$ for $\alpha = 0.4$]{Simulation of $\big\| p^{(N)}_h \big \|_{\infty}$ for $\alpha = 0.4$ and increasing Fourier resolution $N$. 
		}
		\label{figure:line}
	\end{figure}
  
	\subsection{Result for Roughness Parameter $\alpha \in \{0.7, 0.9\}$ } 
	Increasing $\alpha$ is expected to cause a significant increase in the gradient of $v^{(N)}$. Simulations were conducted for $\alpha = 0.7$ and $\alpha = 0.9$. In both cases, the nonlinear remainder~$p^{(N)}_h$ decreases towards zero over the tested resolutions.
	The following presents the data generated for $v^{(N)}$ with $\alpha = 0.7$:
	
	\begin{table}[H]
		\centering
		\begin{tabular}{lrrr}
			\toprule
			Number of Fourier modes & $\left\| v^{(N)} \right \|_{\infty}$ & $\left\| \nabla  v^{(N)} \right \|_{\infty}$  & $\left\| v^{(N)} \right \|_{C^1}$ \\
			\midrule
			${(2^{2})}^2 =16$ & 0.256342 & 1.485884 & 1.742225 \\
			${(2^{3})}^2 =64$ & 0.194115 & 2.131889 & 2.326005 \\
			${(2^{4})}^2 =256$ & 0.147894 & 3.378832 & 3.526726 \\
			${(2^{5})}^2 =1024$ & 0.114401 & 5.150998 & 5.265400 \\
			${(2^{6})}^2 =4096$ & 0.083890 & 7.595636 & 7.679525 \\
			${(2^{7})}^2 =16384$ & 0.060421 & 11.047536 & 11.107957 \\
			${(2^{8})}^2 =65536$ & 0.043793 & 15.591852 & 15.635645 \\
			${(2^{9})}^2 =262144$ & 0.031083 & 22.357294 & 22.388378 \\
			${(2^{10})}^2 =1048576$ & 0.021860 & 31.011626 & 31.033486 \\
			\bottomrule
		\end{tabular} 
	\caption[Norms of the exponential Euler scheme in Fourier space for $\alpha = 0.7$.]{Norms of the exponential Euler scheme in Fourier space with respect to the number of Fourier modes for $\alpha = 0.7$.}
	
	\label{tabelle:3}
	\end{table}
	
	\noindent
	The norm of $\nabla v^{(N)}$ increases gradually as the number of Fourier modes increases. This behaviour is consistent with the suppression of the nonlinear surface current.  
	\noindent
	The graph below no longer exhibits hill-like growth and instead only displays noise.
	\begin{figure}[H]
		\begin{center}
			\includegraphics[scale = 0.5]{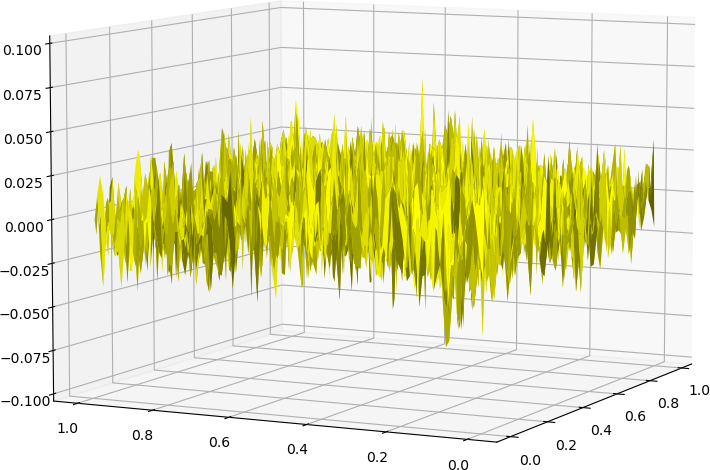}
			\caption[Simulation of $v^{(N)}$ for $\alpha = 0.7$.]{Simulation of $v^{(N)}$ at time $T=10$ for $\alpha = 0.7$ and $N = 2^7$.
			}
			\label{figure:yellow}
		\end{center}
	\end{figure}
	
	\noindent
	The numerical values of $p^{(N)}_h$ decrease towards zero over the tested resolutions:
	
	\begin{table}[H]
		\centering
		\begin{tabular}{lrrr}
			\toprule
			Number of Fourier modes & $\left\| p^{(N)}_h \right \|_{\infty}$ &  $\left\| \nabla p^{(N)}_h \right \|_{\infty}$ &  $\left\|  p^{(N)}_h \right \|_{C^1}$ \\
			\midrule
			${(2^{2})}^2 =16$ & 0.065330 & 0.320001 & 0.385331 \\
			${(2^{3})}^2 =64$ & 0.025951 & 0.149824 & 0.175775 \\
			${(2^{4})}^2 =256$ & 0.009055 & 0.064278 & 0.073333 \\
			${(2^{5})}^2 =1024$ & 0.002977 & 0.025431 & 0.028408 \\
			${(2^{6})}^2 =4096$ & 0.001025 & 0.009572 & 0.010598 \\
			${(2^{7})}^2 =16384$ & 0.000367 & 0.003833 & 0.004200 \\
			${(2^{8})}^2 =65536$ & 0.000146 & 0.001659 & 0.001805 \\
			${(2^{9})}^2 =262144$ & 0.000059 & 0.000690 & 0.000749 \\
			${(2^{10})}^2 =1048576$ & 0.000024 & 0.000286 & 0.000310 \\
			\bottomrule
		\end{tabular}
		\caption[Norms of the nonlinear remainder for $\alpha = 0.7$.]{Norms of the nonlinear remainder with respect to the number of Fourier modes with $\alpha = 0.7$.}
		\label{tabelle:4}
	\end{table}
	
	\noindent
	\Cref{tabelle:4} summarizes the numerical values of the norms of the nonlinear remainder 
	for $\alpha = 0.7$ and increasing numbers of Fourier modes.
	\noindent
	The corresponding decay over the tested resolutions is illustrated in \Cref{figure:line2}. 
		\begin{figure}[H]
			\centering
			\includegraphics[width=0.7\textwidth]{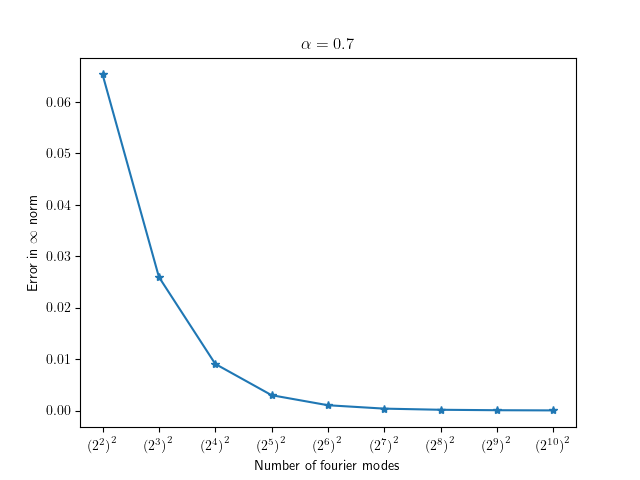} \\
			\caption[Simulation of $\big\| p^{(N)}_h \big \|_{\infty}$ for $\alpha = 0.7$]{Simulation of $\big\| p^{(N)}_h \big \|_{\infty}$ for $\alpha = 0.7$ and increasing Fourier resolution $N$. 
			}
			\label{figure:line2}
	\end{figure}
	
	\noindent
	For $\alpha = 0.9$, we observe the same qualitative behaviour as for $\alpha = 0.7$, but the gradient grows and the nonlinear remainder decays more rapidly over the tested resolutions.
	\begin{table}[H]
	\centering
		\begin{tabular}{lrrr}
			\toprule
			Number of Fourier modes & $\left\| v^{(N)} \right \|_{\infty}$ & $\left\| \nabla  v^{(N)} \right \|_{\infty}$  & $\left\| v^{(N)} \right \|_{C^1}$ \\
			\midrule
			${(2^{2})}^2 =16$ & 0.527842 & 3.426971 & 3.954813 \\
			${(2^{3})}^2 =64$ & 0.562849 & 6.971630 & 7.534479 \\
			${(2^{4})}^2 =256$ & 0.551854 & 14.496256 & 15.048111 \\
			${(2^{5})}^2 =1024$ & 0.548000 & 30.283880 & 30.831880 \\
			${(2^{6})}^2 =4096$ & 0.537246 & 59.610059 & 60.147305 \\
			${(2^{7})}^2 =16384$ & 0.507359 & 113.373272 & 113.880632 \\
			${(2^{8})}^2 =65536$ & 0.476628 & 212.051834 & 212.528462 \\
			${(2^{9})}^2 =262144$ & 0.441572 & 398.740028 & 399.181600 \\
			${(2^{10})}^2 =1048576$ & 0.405952 & 739.753947 & 740.159899 \\
			\bottomrule
		\end{tabular}
	\caption[Norms of the exponential Euler scheme in Fourier space for $\alpha = 0.9$.]{Norms of the exponential Euler scheme in Fourier space with respect to the number of Fourier modes for $\alpha = 0.9$.}
	\label{tabelle:5}
	\end{table}
	 
	\noindent
	The results for the nonlinear remainders with $\alpha = 0.9$ are as follows:
	\begin{table}[H]
		\centering
		\begin{tabular}{lrrr}
			\toprule
			Number of Fourier modes & $\left\| p^{(N)}_h \right \|_{\infty}$ &  $\left\| \nabla p^{(N)}_h \right \|_{\infty}$ &  $\left\|  p^{(N)}_h \right \|_{C^1}$ \\
			\midrule
			${(2^{2})}^2 =16$ & 0.036118 & 0.179939 & 0.216057 \\
			${(2^{3})}^2 =64$ & 0.008669 & 0.060559 & 0.069228 \\
			${(2^{4})}^2 =256$ & 0.001975 & 0.020830 & 0.022805 \\
			${(2^{5})}^2 =1024$ & 0.000615 & 0.007402 & 0.008017 \\
			${(2^{6})}^2 =4096$ & 0.000197 & 0.002494 & 0.002691 \\
			${(2^{7})}^2 =16384$ & 0.000065 & 0.000831 & 0.000896 \\
			${(2^{8})}^2 =65536$ & 0.000021 & 0.000273 & 0.000294 \\
			${(2^{9})}^2 =262144$ & 0.000006 & 0.000084 & 0.000090 \\
			${(2^{10})}^2 =1048576$ & 0.000002 & 0.000026 & 0.000028 \\
			\bottomrule
		\end{tabular}
		\caption[Norms of the nonlinear remainder for $\alpha = 0.9$.]{Norms of the nonlinear remainder with respect to the number of Fourier modes with $\alpha = 0.9$.}
	\label{tabelle:6}
\end{table}

\appendix

	\section{Fourier Series on $\T^2$} \label{subsec:2dTor}

We begin with standard material on Fourier series for the two-dimensional torus, fractional Sobolev spaces, and analytic semigroups generated by the Bilaplace operator. These classical results are detailed in \cite{GIP:15} and \cite{PAZY}.

    \medskip
    \noindent
	Let $\T^2 = [0,L]^2$ be the two-dimensional torus of length $L$.  
	For $u \in L^2([0,T] \times \T^2)$, we consider its Fourier series in space
	\[
	u(t,x) = \sum_{k \in \mathbb{Z}^2} u_k(t)\, e_k(x),
	\]
	for each $x = (x_1,x_2) \in \T^2$,
	where $\{e_k\}_{k \in \mathbb{Z}^2}$ is the standard orthonormal Fourier basis of $L^2(\T^2)$ defined by
	\begin{align} \label{def:ek}
		e_k(x) \coloneqq \omega_{k_1}(x_1)\,\omega_{k_2}(x_2),
		\qquad
		\omega_{k_i}(z) \coloneqq
		\begin{cases}
			\sqrt{\frac{2}{L}} \sin\left( \frac{2\pi k_i}{L}z \right), & k_i > 0,\\
			\frac{1}{\sqrt{L}}, & k_i = 0,\\
			\sqrt{\frac{2}{L}} \cos\left(\frac{2\pi k_i}{L}z\right), & k_i < 0.
		\end{cases}
	\end{align}
	The Fourier coefficients are given by
	\[
	u_k(t) = \langle u(t), e_k \rangle_{L^2(\T^2)}.
	\]

	\noindent
	For each $k = (k_1,k_2)\in \mathbb{Z}^2$, the functions $e_k$ satisfy
	\begin{align} \label{Laplace:EV}
	- \Delta e_k = \mu_k\, e_k,
	\qquad
	\mu_k \coloneqq \left( \frac{2\pi |k|}{L} \right)^2.
	\end{align}
	The Bilaplace operator $A \coloneqq -\Delta^2$ acts diagonally with eigenvalues $-\mu_k^2$:
	\begin{align} \label{def:A}
		A e_k = -\Delta^2 e_k = -\mu_k^2\, e_k.
	\end{align}

	\section{Fractional Sobolev Spaces} \label{sec:fracSobol}
	
	We recall the standard characterization of Sobolev spaces and fractional Sobolev spaces on a Lipschitz domain $\Omega \subset \R^d$.		
		\begin{definition} [Sobolev space]
			For $m\in\mathbb{N}$ and $p \geq 1$ we define
			\[
			W^{m,p}(\Omega)\coloneqq \{u\in L^p(\Omega): D^\beta u\in L^p(\Omega)\ \text{for each}\ |\beta|\leq m\},
			\]
			endowed with the norm
			\[
			\|u\|_{W^{m,p}} \coloneqq \sum_{|\beta|\leq m}\|D^\beta u\|_{L^p}.
			\]
		\end{definition}
		
	
	\noindent
	For simplicity of notation, we set $\ZZ \coloneqq \mathbb{Z}^2 \setminus \{(0,0)\}$.
	
	\begin{definition} \label{def:Halpha}
	For $\alpha \in \R$ we define
	\begin{align} 
		\mathcal{H}^\alpha \coloneqq
		\mathcal{H}^\alpha(\T^2)
		\coloneqq
		\left\{
		u = \sum_{k \in \ZZ} u_k e_k
		: \,
		\sum_{k \in \ZZ} \mu_k^\alpha\, |u_k|^2 < \infty,
		\ \int_{\T^2} u(x)\dd x = 0
		\right\},
	\end{align}
	equipped with norm
	\begin{align} \label{norm:H}
		\|u\|_{\mathcal{H}^\alpha} 
		\coloneqq \|(-\Delta)^{\frac{\alpha}{2}}u\|_{L^2}
		= \left( \sum_{k \in \ZZ} \mu_k^\alpha |u_k|^2 \right)^{\frac{1}{2}}.
	\end{align}
	\end{definition}
	
	\begin{remark}[Poincaré's inequality]
		Due to the moving frame assumption, Poincaré's inequality holds on~$\T^2$, i.e. for each $p \in [1,\infty)$ and $u \in W^{1,p}(\T^2)$ we have
		\[
		\|u\|_{L^p(\T^2)} \leq C \|\nabla u\|_{L^p(\T^2,\R^2)}.
		\]
		In particular, the canonical Sobolev norm $\|\cdot\|_{H^1}$ is equivalent to $\|\cdot\|_{\mathcal{H}^1}$.
	\end{remark}

	\begin{definition}[Vector valued  fractional Sobolev spaces]
		For $\alpha \in \R$, $p \geq 1$ we define the vector-valued fractional Sobolev spaces
		\begin{align*}
			\mathcal{H} ^{\alpha } \left(\T^2, \mathbb{R}^2 \right)
			& \coloneqq 
			\mathcal{H} ^{\alpha } \left(\T^2  \right)
			\times\mathcal{H} ^{\alpha } \left(\T^2 \right)  
		\end{align*}
		equipped with the norms
		\begin{align*}
			\| g \|_{\mathcal{H} ^{\alpha } \left(\T^2, \mathbb{R}^2 \right)}
			& \coloneqq 
            \left(
			\| g_1 \|_{\mathcal{H} ^{\alpha } \left(\T^2  \right)}^2
			+
			\| g_2 \|_{\mathcal{H} ^{\alpha } \left(\T^2  \right)}^2
            \right)^{\frac{1}{2}} 
		\end{align*}
		for $g = \big(g_1 , g_2 \big) \in \mathcal{H} ^{\alpha } \left(\T^2, \mathbb{R}^2 \right)$.
	\end{definition}
	
	\noindent
With this convention, the gradient defines an isometric bounded linear operator
\[
    \nabla\colon
    \mathcal H^{\alpha+1}(\mathbb T^2)
    \longrightarrow
    \mathcal H^{\alpha}(\mathbb T^2;\mathbb R^2),
\]
that is  a direct consequence of integration by parts formula
	\[ \|\nabla y\|_{\mathcal{H}^{\alpha-1}
		\left(
		\T^2, \R^2
		\right)
	}^2 = \langle  (-\Delta) y, y\rangle_{\mathcal{H}^{\alpha-1}} = \|(-\Delta)^{\frac{1}{2}} y\|_{\mathcal{H}^{\alpha-1}}^2
	=  \| y\|_{\mathcal{H}^\alpha}^2
	\]
	for $y \in \mathcal{H}^\alpha$. 
	In particular, for the divergence we obtain 
	\[
	\|\nabla \cdot g\|_{L^2} = \| \partial_x g_1+ \partial_y g_2\|_{L^2} \leq C \|g\|_{\mathcal{H}^1\left( \T^2, \R^2 \right)},
	\]
	where $g = \big(g_1 , g_2 \big) \in \mathcal{H}^1\left( \T^2, \R^2 \right)$.
	
	\section{Analytic Semigroup Generated by the Bilaplace Operator}
	
	The negative Bilaplace operator $A=-\Delta^2$ with periodic boundary conditions generates an analytic semigroup $\big( e^{tA} \big)_{t \geq 0}$ on $L^p(\T^2)$ for every $p \in (1,\infty)$, see \cite[pp.~212--214]{PAZY}.  
	Using the Fourier representation,
	\begin{align} \label{Fourier:etA}
		e^{tA} u(x) = \sum_{k \in \ZZ} e^{-t \mu_k^2} u_k e_k(x),
		\qquad t \geq 0,\ x \in \T^2.
	\end{align}
	
	\begin{lemma} \label[lemma]{lem:semigroupestimate}
		For $\beta > \alpha$ and $t>0$, the following estimate holds:
		\begin{align} \label{e:SG}
			\left\| e^{tA} \right\|_{L\left(\mathcal{H}^\alpha,\mathcal{H}^\beta\right)}
			\leq \left( \frac{\beta - \alpha}{4e} \right)^{\frac{\beta - \alpha}{4}} t^{\frac{\alpha - \beta}{4}}.
		\end{align}
	\end{lemma}
	
	\noindent
	This result follows directly from basic calculus applied to the Fourier expression and the eigenvalues of $-\Delta^2$.
	For a more detailed overview, see \cite[Lemma~A.7]{GIP:15}.

		\noindent	
		Thus,
		$$e^{tA}\nabla  \cdot  \colon
		\mathcal{H}^{\alpha -1} \left(\T^2, \mathbb{R}^2  \right) \rightarrow \mathcal{H}^{\alpha} \left(\T^2\right)
		$$
		is a bounded linear operator with 
		\[
		\left \|
		e^{tA} \nabla \cdot
		\right \|_{L \left( \mathcal{H}^{\alpha-1} \left(\T^2 , \R^2 \right),  \mathcal{H}^\alpha \right)}
		\leq 
		\left \|
		e^{tA}
		\right \|_{L \left( \mathcal{H}^{\alpha-2},  \mathcal{H}^\alpha \right)}
		\left \|
		\nabla \cdot\right \|_{L \left( \mathcal{H}^{\alpha-1} \left(\T^2 , \R^2 \right),  \mathcal{H}^{\alpha-2} \right)}
		\leq 
		C   t^{-\frac{1}{2}}.
		\]

\section{Technical Results}
\begin{lemma} [{\cite[Theorem~2.21]{Rimmele2026}}] \label[lemma]{lem:Z-C0}
Let $\varepsilon>0$, $\psi \in(0,\frac12)$, and $p\geq1$. 
Then the stochastic convolution~$Z$ and the regularized stochastic convolution~$P_N Z$ belong to $
L^p\bigl(\Omega;C^0([0,T]\times\T^2)\bigr)$,
and there exists a constant $C=C(p,L,T)>0$ such that
\begin{align*}
\left( \mathbb{E} \|Z\|_{C^0([0,T]\times\T^2)}^p
\right)^{1/p}
& \leq
C
\left(
\sum_{k\in\ZZ}
 \mu_k^{2\psi-2}
\right)^{1/2}
<\infty
\\
\left( \mathbb{E}
\| P_N Z \|_{C^0([0,T]\times\T^2)}^p
\right)^{1/p}
& \leq
C
\left(
\sum_{k\in\ZZ, |k | \leq N}
 \mu_k^{2\psi-2}
\right)^{1/2}
<\infty.
\end{align*}

\end{lemma}

\begin{proof}
This follows from \cite[Theorem~2.21]{Rimmele2026}.
\end{proof}

The following lemma from \cite[Lemma~7.1.1]{HDG} provides an explicit upper bound for the growth rate of solutions, even in the presence of a singularity at time $t$ within the integrand.	
	\begin{lemma}[Henry--Grönwall lemma]\label[lemma]{lem:G-HDG}
		Suppose $b\geq 0$, $\beta>0$ and $a$ is a non-negative
		function locally integrable on $[0,T]$ (some $ T\leq+\infty$), and suppose
		$f$ is nonnegative and locally integrable on $0\leq t \leq  T$
		with 
		\begin{align*}
			f(t)\leq a(t)+b\int_0^t (t-s)^{\beta-1}f(s)ds
		\end{align*}
		for each $0\leq t  < T$.
		Then \begin{align}
			\label {est-Gronw}
			f(t)\leq a(t)+\theta\int_0^t E_\beta'(\theta(t-s))a(s)ds,
		\end{align}
		holds for each $0 \leq t < T$,
		where 
		\begin{align*}
			\theta=b(\Gamma(\beta))^\frac{1}{\beta},  \quad E_\beta=\sum_{n=0}^\infty \frac{z^{n\beta}}{\Gamma(n\beta+1)},  \quad  E_\beta'=\frac{\partial}{\partial z}E_\beta(z).
		\end{align*}
		Particularly, if $a(t)\equiv a$ is constant, we obtain 
		\begin{align*}
			f(t)\leq a  E_\beta(\theta t).
		\end{align*}
	\end{lemma}
	
	\begin{remark}
		In order to obtain \Cref{lem:G-HDG} on $[0,T]$ instead of $[0,T)$, we must require that $f$ is continuous at $T$.
	\end{remark}

\begin{lemma}\label[lemma]{lem:semigroupincrement}
Let $A \coloneqq - (-\Delta)^2$.
For every $s\in\R$, $\theta\in(0,1]$, $u\in\mathcal H^{s+4\theta}$, and $t\geq0$, 
\begin{align}\label{ineq:SGalpha}
\|(e^{tA}-\operatorname{Id})u\|_{\mathcal H^s}
\leq
C_\theta t^\theta
\|u\|_{\mathcal H^{s+4\theta}}.
\end{align}
\end{lemma}

\begin{proof}
    The assertion follows from Parseval's identity, and
$$|1-e^{-x}|\leq C_\theta x^\theta
\qquad \text{for $x\geq0$.}
$$
\end{proof}
\pagebreak

\paragraph{Acknowledgements} 
All three  authors acknowledge the support of DFG BL 535/12-1, Project number: 514726621

\end{document}